\documentclass[12pt,twoside,a4paper]{amsart}
\usepackage{amssymb}
\usepackage{latexsym}
\usepackage{amsmath}
\usepackage{euscript}
\usepackage{graphics}

\newcommand{\ba}{\begin{array}}
\newcommand{\ea}{\end{array}}
\newcommand{\bea}{\begin{eqnarray}}
\newcommand{\eea}{\end{eqnarray}}
\newcommand{\bead}{\begin{eqnarray*}}
\newcommand{\eead}{\end{eqnarray*}}
\newcommand{\be}{\begin{equation}}
\newcommand{\ee}{\end{equation}}
\newcommand{\bed}{\begin{displaymath}}
\newcommand{\eed}{\end{displaymath}}
\newcommand{\bl}{\begin{lemma}}
\newcommand{\el}{\end{lemma}}
\newcommand{\bp}{\begin{proposition}}
\newcommand{\ep}{\end{propostion}}
\newcommand{\bt}{\begin{theorem}}
\newcommand{\et}{\end{theorem}}

\newcommand{\bc}{\begin{corollary}}
\newcommand{\ec}{\end{corollary}}

\newcommand{\br}{\begin{remark}}
\newcommand{\er}{\end{remark}}
\newcommand{\bd}{\begin{definition}}
\newcommand{\ed}{\end{definition}}
\newcommand{\bex}{\begin{example}}
\newcommand{\eex}{\end{example}}

\newcommand\R{{\mathbb{R}}}
\newcommand\C{{\mathbb{C}}}

\newcommand\N{{\mathbb{N}}}

\def\st{{\mathfrak t}}

      \def\dC{{\mathbb C}}

      \def\dR{{\mathbb R}}

   \def\cB{{\mathcal B}}   \def\cC{{\mathcal C}}
\def\cD{{\mathcal D}}      \def\cF{{\mathcal F}}
   \def\cH{{\mathcal H}}   
      \def\cL{{\mathcal L}}

\def\cS{{\mathcal S}}

\def\wt#1{{{\widetilde #1} }}
\def\wh#1{{{\widehat #1} }}

\def\bm\chi{\mbox{\boldmath$\chi$}}

\def\RE{{\rm Re\,}}
\def\IM{{\rm Im\,}}

\def\ran{{\rm ran\,}}

\def\dom{{\rm dom\,}}
\def\mul{{\rm mul\,}}

\def\cdom{{\rm \overline{dom}\,}}
\def\clos{{\rm clos\,}}

\def\dim{{\rm dim\,}}

\let\xker=\ker \def\ker{{\xker\,}}
\def\span{{\rm span\,}}

\def\cmr{{\dC \setminus \dR}}
\def\uphar{{\upharpoonright\,}}

\DeclareMathOperator{\hplus}{\, \widehat + \,}

\newtheorem{theorem}{Theorem}[section]
\newtheorem{proposition}[theorem]{Proposition}
\newtheorem{corollary}[theorem]{Corollary}
\newtheorem{lemma}[theorem]{Lemma}

\theoremstyle{definition}
\newtheorem{example}[theorem]{Example}
\newtheorem{remark}[theorem]{Remark}
\newtheorem{definition}[theorem]{Definition}

\numberwithin{equation}{section}

\begin{document}

\title[Nevanlinna families]
{Invariance theorems for Nevanlinna families}
\author{Vladimir Derkach}
\author{Seppo Hassi}
\author{Mark Malamud}


 \address{Department of Mathematics, National Pedagogical University,
Kiev, Pirogova 9, 01601, Ukraine}
\email{derkach.v@gmail.com}
\address{Department of Mathematics and Statistics \\
University of Vaasa \\
P.O. Box 700, 65101 Vaasa \\
Finland}
\email{sha@uwasa.fi}
\email{malamud3m@yahoo.com }




\begin{abstract} A complex function $f(z)$ is called a Herglotz-Nevanlinna function if
it is holomorphic in the upper half-plane ${\mathbb C}_+$ and maps
${\mathbb C}_+$ into itself. By a maximum principle a
Herglotz-Nevanlinna function which takes a real value $a$ in a
single point $z_0\in {\mathbb C}_+$ should be identically equal to
$a$. In the present note we prove similar invariance results both
for the point and the continuous spectra of an operator-valued
Herglotz-Nevanlinna function with values in the set of bounded or
unbounded linear operators (or relations) in a Hilbert space. The
proof of this invariance result for continuous spectrum is based on
Harnack's inequality. This inequality is systematically used to
characterize operator-valued Herglotz-Nevanlinna functions with
form-domain invariance property for their imaginary parts or
Herglotz-Nevanlinna functions with values in the Schatten-von
Neumann classes.
\end{abstract}

\maketitle


\section{Introduction}
The class of Herglotz-Nevanlinna functions plays an important role
in function theory, probability theory, mathematical physics, etc.
In particular, the m-function of a Sturm-Liouville operator on a
half-line belongs to this class; \cite{Titch62}, \cite{Co}.
Similarly, the M-function of an elliptic operators is an
operator-valued Herglotz-Nevanlinna function; see \cite{AmrPear04}.
Also the Kre\u{\i}n's formula for (generalized) resolvents involves
an another source for various important applications of this class.
In particular, Kre\u{\i}n's formula allows to parametrize sets of
solutions of various classical interpolation and moment problems
with a parameter ranging over the class of Herglotz-Nevanlinna
families; cf. \cite{Kr46}, \cite{KL71}. For basic properties of
Herglotz-Nevanlinna functions see e.g. the surviews in~\cite{KacK},
\cite{Br}, \cite{Don74}.

The class $R[\cH]$ of Herglotz-Nevanlinna functions with values in the set $\cB(\cH)$
of bounded linear operators in a separable Hilbert space $\cH$ is defined as follows.
\begin{definition}
\label{def:Rfunc0} An operator-valued function  $F({z})$ holomorphic
on $ \cmr$, with values in $\cB(\cH)$ is said to belong to the class
$R[\cH]$, if:
\begin{enumerate}
\def\labelenumi{\rm (\roman{enumi})}
\item for every ${z} \in \dC_+ (\dC_-)$
      the operator $F({z})$ is  dissipative
      (resp. accumulative);
\item $F({z})^*=F(\bar {z})$, ${z} \in \cmr$;
\end{enumerate}
\end{definition}
In what follows an operator $T\in\cB(H)$ is called dissipative
      (resp. accumulative), if its  imaginary part
\[
\IM (T)=\frac{1}{2i}(T-T^*)
\]
is a nonnegative (resp. nonpositive) operator in $\cH$; cf.
\cite{Kato}, \cite{Ph0}.

Each operator-valued function $F\in R[\cH]$ admits the following integral
representation
\begin{equation}
\label{INTrep}
 F({z})
  = B_0 + B_1{z}+\int_{\dR}\left(\frac{1}{t-{z}}-\frac{t}{t^2+1}\right)\,
               d\Sigma(t),
\end{equation}
where $B_0= B_0^*\in\cB(\cH)$, $0\le B_1=B_1^*\in\cB(\cH)$, and
$\Sigma(\cdot)$ is a $\cB(\cH)$-valued operator  measure, such that
  \begin{equation}
\label{INT_cond} K_\Sigma :=
\int_{\dR}\,\frac{d\Sigma(t)}{t^2+1}\in\cB(\cH).
 \end{equation}
Here  integral in~\eqref{INT_cond} is uniformly convergent in
$\cB(\cH)$; cf. \cite{Br}, \cite{KacK},  \cite{GesTs00}.

The next result summarizes some invariance results on the spectra
properties of operator-valued functions $F\in R[\cH]$
(cf.~\cite[Proposition~1.2]{DM97}).

\begin{theorem}\label{prop:Inv}
Let $F\in R[\cH]$, $z_0\in\dC_+$ and $a=\bar a$. Then the following equivalences hold:
\begin{enumerate}
  \item [(1)] $0\in\sigma_p(\IM(F(z_0)))\Longleftrightarrow 0\in\sigma_p(\IM(F(z)))$ for all $z\in\dC_+$;
  \item [(2)] $0\in\sigma_c(\IM(F(z_0)))\Longleftrightarrow 0\in\sigma_c(\IM(F(z)))$ for all $z\in\dC_+$;
    \item [(3)] $0\in\rho(\IM(F(z_0)))\Longleftrightarrow 0\in\rho(\IM(F(z)))$ for all $z\in\dC_+$;
  \item [(4)] $a\in\sigma_p(F(z_0))\Longleftrightarrow a\in\sigma_p(F(z))$ for all $z\in\dC_+$;
  \item [(5)] $a\in\sigma_c(F(z_0))\Longleftrightarrow a\in\sigma_c(F(z))$ for all $z\in\dC_+$;
  \item [(6)]  $a\in\rho(F(z_0))\Longleftrightarrow a\in\rho(F(z))$ for all $z\in\dC_+$.
\end{enumerate}
\end{theorem}

The following two subclasses of the class $R[\cH]$ appear in the
theory of $Q$-functions of symmetric operators, \cite{KL73}, and in
the boundary triplet approach to the extension theory of symmetric
operators, \cite{DM91,DM95}.
\begin{equation}\label{eq:R_su}
    \begin{split}
 R^s[\cH]&=\left\{F(\cdot) \in R[\cH]:\,\ker \IM F({z})=\{0\}\mbox{  for all
         }{z}\in\cmr\,\right\};\\
 R^u[\cH]&=\left\{F(\cdot) \in R^s[\cH]:\,0 \in \rho(\IM F({z}))\mbox{  for all
         }{z}\in\cmr\,\right\}.
\end{split}
\end{equation}

It follows from Theorem~\ref{prop:Inv} that each of the subclasses $R^s[\cH]$ and
$R^u[\cH]$ can be single out  by a single condition:
\begin{equation}\label{eq:R_su2}
    \begin{split}
 R^s[\cH]&=\left\{F \in R[\cH]:\,\ker \IM F(i)=\{0\}\,\right\};\\
 R^u[\cH]&=\left\{F \in R^s[\cH]:\,0 \in \rho(\IM F(i))\,\right\}.
\end{split}
\end{equation}
The classes  $R^u[\cH]$, $R^s[\cH]$, and $R[\cH]$ are ordered by
inclusion
\begin{equation}
R^u[\cH] \subset R^s[\cH] \subset R[\cH].
\end{equation}

It follows from~\eqref{eq:R_su2} and Theorem~\ref{prop:Inv} that the
operator function $F(z)$ with the integral
representation~\eqref{INTrep} belongs to the class $R^s[\cH]$ (or
$R^u[\cH]$), if $0\not\in\sigma_p(\Sigma_0)$ ($0\in\rho(\Sigma_0)$,
respectively).

Recall that every Nevanlinna family can be realized (uniquely up to unitary equivalence)
as a Weyl function (or Weyl family) of a (minimal unitary) boundary relation; see
\cite{DHMS06}, \cite[Theorem~3.9]{DHMS12}. Various other subclasses of the class $\wt
R(\cH)$ appearing above and below can  be characterized in boundary triplet and boundary
relation context. 
Each property of $R$-function in Theorem~1.2 when treated as a Weyl function of a
(generalized) boundary triplet has its geometrical counterpart. For
instance,  the class of $R^u[\cH]$-functions  is known to characterize  the class of Weyl functions
corresponding to ordinary boundary triplets of $A^*$, where $A$ is a
not necessarily densely defined symmetric operator in $\mathfrak H$
(see \cite{KL73}, \cite{DM91,DM95}).
 The class $ R^s[\cH]$  gives a  precise characterization of
 Weyl functions of 
symmetric operators, corresponding to the so-called $B$-generalized boundary triplets  (see \cite{DM95}, \cite{DHMS06}).

The paper is organized as follows. For later purposes a proof of
items (1) -- (4) in Theorem~\ref{prop:Inv} will be presented in
Section~2, while items (5) and (6) will be treated in Section~3,
where all of these invariance results are extended to the class
$R(\cH)$ of Herglotz-Nevanlinna functions with values in the set
$\cC(\cH)$ of closed linear operators and to the class $\wt R(\cH)$
of Nevanlinna families. The proof of the first half of these
invariance results is based on the maximum principle for
Herglotz-Nevanlinna functions or alternatively for contractive
holomorphic operator functions. The rest is proven then with the
help of the Harnack's inequality for harmonic functions. It is
emphasized that no realization results via operator or functional
models, or boundary triplets methods, for functions from these
classes of operator functions are involved in the given arguments.

Harnack's inequality is systematically used in Section~4 to
characterize invariance properties of operator-valued harmonic
functions as well as for Herglotz-Nevanlinna functions whose
imaginary parts have a so-called form-domain invariance property and
for Herglotz-Nevanlinna functions with values in the Schatten-von
Neumann classes. For instance, by applying such analytic arguments
it is shown (see Proposition \ref{prop6.2}) that under certain
additional assumptions a Herglotz-Nevanlinna function $F(\cdot)$
whose imaginary part is bounded at one point admits a representation
     \begin{equation}\label{6.9B}
F(z) =  G(z) + T, \quad z\in \C_+, 
     \end{equation}
where $G(\cdot)\in R[\cH]$ and   $T=T^*\in \cC(\cH)$ if and only if $F_I(z_0)\in
\cB(\cH)$ for some $z_0\in\mathbb C_+$. Functions of the form  \eqref{6.9B} with
$G(\cdot)\in R^s[\cH]$ characterize Weyl functions of generalized boundary triplets with
a selfadjoint operator $A_0= A^*\lceil \ker\Gamma_0$; see \cite[Section~4]{DHMS06} ,
\cite[Theorem 7.39]{DHMS12}. Similar functions appear also in a  connection with the
so-called quasi-boundary triplets introduced and investigated  in \cite{BeLa07}. In
fact, an arbitrary function of the form,
\[
 F(z)=T+F_0(z), \quad z\in \cmr,
\]
where $T$ is a symmetric densely defined operator (not necessarily self-adjoint or even
closed) on $\cH$ and $F_0(\cdot)$ belongs to the class $R[\cH]$. Such $F(\cdot)$ appears
as a Weyl function of a so-called almost $B$-generalized boundary triplet (possibly
multi-valued); a concept that is originated  in a forthcoming paper \cite{DHM15} by the
authors. Such functions play a central role in the study of form-domain invariant
Nevanlinna families (see Definition \ref{Nevforminv} below). Characteristic properties
of such functions are investigated in \cite{DHM15} within boundary triplet (and boundary
relation) setting: in particular, it is shown therein that a Herglotz-Nevanlinna
function having a form-domain invariant imaginary part can be realized as the Weyl
function of a unitary boundary triplet (boundary relation) $\{\cH,\Gamma_0,\Gamma_1\}$
with an essentially selfadjoint kernel $A_0=\ker \Gamma_0$. Such functions appear in
applications, e.g. in the study of local point interactions, \cite{KosMMM,MN2012}, in
PDE setting as M-functions, Dirichlet-to-Neumann maps, and their analogs; see e.g.
\cite[Section 7.7]{Post12} for a treatment of the Zaremba problem.

Finally, in Section~5 we present several examples which reflect various invariance
properties of associated Herglotz-Nevanlinna functions and stability properties of
quadratic forms generated by the imaginary parts of such functions.

\section{Preliminaries}

\subsection{Linear relations in Hilbert spaces}
Let $\cH$ be a separable Hilbert space.
The set of bounded (closed) linear operators in $\cH$ is denoted by $\cB(\cH)$,
($\cC(\cH)$, respectively).

Recall that a linear relation $T$ in $\cH$ is a linear subspace of
$\cH \times \cH$. Systematically a linear operator $T$ will be
identified with its graph. The set of closed linear relations in
$\cH$ is denoted by $\wt \cC(\cH)$. It is convenient to  interpret
the linear relation $T$ as a multi-valued linear mapping from $\cH$
into $\cH$. For a linear relation $T\in\wt \cC(\cH)$ the symbols
$\dom T$, $\ker T$, $\ran T$, and $\mul T$ stand for the domain,
kernel, range, and the multi-valued part, respectively. The inverse
$T^{-1}$ is a linear relation in $\cH$ defined by
$\{\,\{f',f\}:\,\{f,f'\}\in T\,\}$. The adjoint $T^*$ is the closed
linear relation from $\cH$ defined by (see~\cite{AI}, \cite{Ben},
\cite{Co})
\[
 T^*=\{\,\{h,k\} \in \cH \oplus \cH :\,
     (k,f)_{\cH}=(h,g)_{\cH}, \, \{f,g\}\in T \,\}.
\]
The sum $T_1+T_2$ and the componentwise sum
$T_1 \wh + T_2$ of two linear relations
$T_1$ and $T_2$ are defined by
\[
\begin{split}
  &T_1+T_2=\{\,\{f, g+h\} :\, \{f,g\} \in T_1, \{f,h\} \in T_2\,\}, \\
  &T_1 \hplus T_2=\{\, \{f+h,g+k\} :\, \{f,g\} \in T_1, \{h,k\}\in
T_2\,\}.
\end{split}
\]
If the componentwise sum is orthogonal it will be denoted by
$T_1 \oplus T_2$.
Moreover, $\rho(T)$ ($\hat\rho(T)$) stands for the set of regular
(regular type) points of $T$. The closure of a linear relation $T$
will be denoted by $\clos T$.

 Recall that a linear relation $T$ in $\cH$ is called
\textit{symmetric} (\textit{dissipative} or \textit{accumulative})
if $\mbox{Im }(h',h)=0$ ($\ge 0$) or $\le 0$, respectively) for all
$\{h,h'\}\in T$. These properties remain invariant under closures.
By polarization it follows that a linear relation $T$ in $\cH$ is
symmetric if and only if $T \subset T^*$. A linear relation $T$ in
$\cH$ is called \textit{selfadjoint} if $T=T^*$, and it is called
\textit{essentially selfadjoint} if $\clos T=T^*$. A dissipative
(accumulative) linear relation $T$ in $\cH$ is called maximal
dissipative (maximal accumulative) if it has no proper dissipative
(accumulative) extensions.

Assume that $T$ is closed. If $T$ is dissipative or accumulative,
then $\mul T \subset \mul T^*$.
In this case the orthogonal decomposition
$\cH=(\mul T)^\perp \oplus \mul T$
induces an orthogonal decomposition of $T$ as
\[
 T=T_s \oplus T_\infty, \quad T_\infty=\{0\} \times \mul T, \quad
 T_s=\{\, \{f,g\} \in T:\, g \perp \mul T\,\},
\]
where $T_\infty$ is a selfadjoint relation in $\mul T$
and $T_s$ is an operator in $\cH \ominus \mul T$
with $\cdom T_s=\cdom T=(\mul T^*)^\perp$,
which is dissipative or accumulative.
Moreover, if the relation $T$ is maximal
dissipative or accumulative, then $\mul T=\mul T^*$.
In this case the orthogonal decomposition
$(\dom T)^\perp=\mul T^*$ shows that $T_s$
is a densely defined dissipative or accumulative
operator in $(\mul T)^\perp$, which is maximal
(as an operator); see e.g. \cite[Sec.~3,~Cor.~4.16 ]{HdSSz2009}.
In particular, if $T$ is a selfadjoint relation,
then there is such a decomposition where $T_s$
is a selfadjoint operator
(densely defined in $(\mul T)^\perp$).

%
%

\subsection{Nevanlinna families}

\begin{definition}
A family of linear relations $\cF({z})$, ${z} \in \cmr$, in
a Hilbert space $\cH$ is called a \textit{Nevanlinna family} if:
\begin{enumerate}
\item[(NF1)] for every ${z} \in \dC_+ (\dC_-)$
      the relation $\cF({z})$ is maximal dissipative
      (resp. accumulative);
\item[(NF2)] $\cF({z})^*=\cF(\bar {z})$, ${z} \in \cmr$;
\item[(NF3)] for some (and hence for all) ${w} \in \dC_+ (\dC_-)$ the
      operator family $(\cF({z})+{w})^{-1} (\in [\cH])$ is
      holomorphic in ${z} \in \dC_+ (\dC_-)$.
\end{enumerate}
\end{definition}


The condition $(\cF({z})+{w})^{-1} (\in [\cH])$ implies that $\cF({z})$ is maximal dissipative
      or accumulative relation $\cF({z})$, ${z}
\in \cmr$, and thus, in particular, closed. The \textit{class of all
Nevanlinna families} in a Hilbert space is denoted by $\wt
R(\cH)$. If the multi-valued part $\mul \cF({z})$ of $\cF(\cdot)
\in \wt R(\cH)$ is nontrivial, then it is independent of ${z}
\in \cmr$, so that
\begin{equation} \label{ml}
 \cF({z})=\cF_s({z}) \oplus \cF_\infty, \quad \cF_\infty=\{0\}
 \times \mul \cF({z}), \quad {z}\in\dC\setminus\dR,
\end{equation}
where $\cF_s({z})$ is a Nevanlinna family of single-valued linear relations in $\cH \ominus \mul F({z})$, \cite{KL71}.

Clearly, if $\cF(\cdot) \in \wt R(\cH)$, then $\cF_\infty \subset
\cF({z}) \cap \cF({z})^*$ for all ${z}\in\dC\setminus\dR$.

\subsection{Class $R(\cH)$}
Class $R(\cH)$ is defined below as a class of all single-valued Nevanlinna families.
\begin{definition}
\label{def:Rfunc}
An operator-valued function  $F({z})$, ${z} \in \cmr$, with values in $\cC(H)$ is said to belong to the class $R(\cH)$, if:
\begin{enumerate}
\item[(NF1)] for every ${z} \in \dC_+ (\dC_-)$
      the operator $F({z})$ is maximal dissipative (resp. accumulative);
\item[(NF2)] $F({z})^*=F(\bar {z})$, ${z} \in \cmr$;
\item[(NF3)] for some (and hence for all) ${w} \in \dC_+ (\dC_-)$ the
      operator function $(F({z})+{w})^{-1} (\in [\cH])$ is
      holomorphic in ${z} \in \dC_+ (\dC_-)$.
\end{enumerate}
\end{definition}

The following subclass
\[
R[\cH]=\left\{F(\cdot) \in  R(\cH):\,\dom F({z})=\cH\mbox{  for all }{z}\in\cmr\,\right\}
\]
of the class $R(\cH)$ consist of operator-valued functions $F({z})$ with values in $\cB(\cH)$.

Analogous to~\eqref{eq:R_su}, the following subclasses of the class $R(\cH)$ can be determined
\[
\begin{split}
 R^s(\cH)&=\left\{F(\cdot) \in  R(\cH):\ker(\IM F({z})) =\{0\}\mbox{ for all
         }{z}\in\cmr\,\right\};\\
 R^u(\cH)&=\left\{F(\cdot) \in R(\cH):\,\cF({z}) \hplus \cF({z})^*=\cH^2\mbox{  for all
         }{z}\in\cmr\,\right\},\\
\end{split}
\]
where $\cF(z)$ is the graph of the linear operator $F(z)$.
As will be shown in Proposition~\ref{Runif+} (see also~\cite[Proposition~2.18]{DHMS06}) the classes  $R^u(\cH)$ and $R^u[\cH]$ coincide.
 The
Nevanlinna functions in $R^s(\cH)$ and $R^u[\cH]$ will be called
\textit{strict} and \textit{uniformly strict}, respectively.

These definitions
give rise to the following chain of inclusions:
\begin{equation}
\label{scheme}
R^u(\cH) \subset R^s(\cH) \subset R(\cH)  \subset \wt R(\cH).
\end{equation}
In the infinite-dimensional situation each of the inclusions in~\eqref{scheme}
is strict.

It follows from the integral representation~\eqref{INTrep} for every $F\in R[\cH]$ the kernel
\begin{equation}\label{eq:Kern_F}
    {\mathbf N}_F(z,w):=\left\{\begin{array}{cc}
                          {\displaystyle\frac{F(z)-F( w)^*}{z-\bar w}}, & \mbox{if }\, z\ne\bar w;  \\
                          F'(z) &  \mbox{if }\, z=\bar w.
                        \end{array}\right.
\end{equation}
is nonnegative in $\dC_+\cup\dC_-$. This observation leads to a couple of the invariance results in
the class $R[\cH]$, which were stated in Theorem~\ref{prop:Inv} and can be considered to be well known.
For the convenience of the reader, a short proof for items (1)-(4) is given;
the rest will follow from a more general Theorem~\ref{invar} given below.

\begin{proof}[Proof of Theorem~\ref{prop:Inv} (1)-(4)]
(1) Assume that $0\in\sigma_p(\IM(F(z_0)))$ and $\IM(F(z_0))h_0=0$ for some $h_0\in\cH$, $h_0\ne 0$.
The matrix
\begin{equation}\label{eq:Im_F0}
 \left(
      \begin{array}{cc}
        \left({\mathbf N}_F(z_0,z_0)h_0,h_0\right) & \left({\mathbf N}_F(z_0,z)h_0,h_0\right) \\
        \left({\mathbf N}_F(z_0,z_0)h_0,h\right) & \left({\mathbf N}_F(z,z)h_0,h_0\right) \\
      \end{array}
    \right)
 \end{equation}
is nonnegative for all $z\in\dC_+\cup\dC_-$, $z\ne \bar z_0$. Since
the left-upper corner of this matrix vanishes, then also
$\left({\mathbf N}_F(z,z_0)h_0,h_0\right)=0$ and hence
\begin{equation}\label{eq:Im_F2}
(F(z)h_0,h_0)=(F(z_0)^*h_0,h_0)=(h_0,F(z_0)h_0) \quad\mbox{for
all}\quad z\in\dC_+\cup\dC_-\,\,(z\ne \bar z_0).
 \end{equation}
Hence, $(\IM(F(z))h_0,h_0)=0$ and since $\IM(F(z))\geq 0$ (or $\le
0)$ this gives $\IM(F(z))h_0=0$.

(2) Assume that $0\in\sigma(\IM(F(z_0)))$ and $\IM(F(z_0))h_n\to 0$
as $n\to\infty$ for some sequence $h_n\in\cH$, such that
$\|h_n\|=1$. Then for all $z\in\dC_+\cup\dC_-$, $z\ne \bar z_0$,
\[
  \left|({\mathbf N}_F(z,z_0)h_n,h_n)\right|^2\le  \left({\mathbf N}_F(z_0,z_0)h_n,h_n\right)
  \left({\mathbf N}_F(z,z)h_n,h_n\right)\to 0 \quad (n\to\infty).
\]
This implies that $(\IM(F(z))h_n,h_n)\to 0$. Hence,
$\|(\IM(F(z))^{1/2}h_n\|^2\to 0$ (if e.g. $\IM z>0$) and therefore
also $\IM(F(z))h_n\to 0$ as $n\to\infty$.

(3) This statement is implied by (1) and (2).

(4) Assume that $F(z_0)h_0=ah_0$ for some $h_0\in\cH$, $h_0\ne 0$.
Then the left--upper corner of the matrix in~\eqref{eq:Im_F0} equals
to 0 and therefore~\eqref{eq:Im_F2} holds. Hence one obtains
$F(z)h_0=ah_0$ for all $z\in\cmr$.

The proof of (5) and (6) is postponed until Theorem~\ref{invar} (v), (vi).
\end{proof}

\section{Nevanlinna pairs}
In abstract eigenvalue depending boundary value problems Nevanlinna
family is often represented via its counterpart -- a Nevanlinna pair,
see e.g. \cite{DM95}, \cite{CDR}, \cite{DHMS1}. In this section
connections between Nevanlinna families and Nevanlinna pairs are
investigated in the general Hilbert space setting.

Every closed linear relation $T$ in a separable Hilbert space $\cH$
can be represented as
\begin{equation}
\label{tau00}
 T=\{\,\{\Phi h,\Psi h\}:\, h\in\cL \,\},
\end{equation}
where $\cL$ is a parameter Hilbert space and the operators $\Phi$,
$\Psi$ belong to $[\cL,\cH]$. To show this it is enough to take $T$
as $\cL$ and the projections $\pi_1$, $\pi_2$ onto the first and the
second components of $T\subset\cH\times\cH$ as $\Phi$ and $\Psi$.
Clearly, each pair $\{\Phi,\Psi\}$ of operators in $[\cL,\cH]$ gives
rise to a linear relation $T$ in $\cH$ via~\eqref{tau00}. In the
infinite-dimensional case ($\dim\cH=\infty$) the parameter Hilbert
space $\cL$ can be taken to be equal to $\cH$. Note that when
$\rho(T)$ is not empty and ${z}_0 \in \rho(T)$ then
\[
 T=\{\,\{(T-{z}_0)^{-1}h, (I+{z}_0(T-{z}_0)^{-1})h\} :\,
     h \in \cH\,\},
\]
so that $\cL=\cH$ and there is a natural choice for the pair
$\{\Phi, \Psi\}$ in $\cB(\cH)$.

For linear relations given by the equation~\eqref{tau00} its
properties 
can be characterized
in terms of the pair $\{\Phi, \Psi\}$.

\begin{proposition}[\cite{DHMS1}]\label{CRIT}
Let $T$ be a linear relation $T$ in $\cH$, defined by~\eqref{tau00}.
Then:
 \begin{enumerate}
\def\labelenumi{\rm (\roman{enumi})}
\item
the adjoint $T^*$ is a linear relation given by
\begin{equation}
\label{tau01}
 T^*=\{\,\{h,h'\}\in\cH^2:\, \Psi^*h-\Phi^*h'=0 \,\}.
\end{equation}
\item
$T$ is a dissipative  (accumulative) relation if and only if
\begin{equation}
\label{Acor1}
 -i(\Phi^*\Psi- \Psi^*\Phi)\ge 0, \quad ( \le 0);
\end{equation}
\item
$T$ is symmetric if and only if
\begin{equation}
\label{cor2}
 \Phi^*\Psi- \Psi^*\Phi=0;
\end{equation}
\end{enumerate}
If, additionally, $\ker \Phi\cap\ker\Psi=\{0\}$, then
 \begin{enumerate}
\def\labelenumi{\rm (\roman{enumi})}
\item[(iv)]
${z}\in\rho(T)$ if and only if the operator $\Psi-{z}\Phi$
has a bounded inverse;
\item[(v)]
$T$ is  maximal dissipative (accumulative) if and only
if~\eqref{Acor1} holds and the operator $\Psi+i \Phi$ ($\Psi-i
\Phi$) has a bounded inverse;
\item[(vi)]
$T$ is selfadjoint  if and only if~\eqref{cor2} holds and the
operators $\Psi\pm i \Phi$  have bounded inverses.
\end{enumerate}
\end{proposition}

\begin{proof}
(i) For $\{g,g'\}\in T^*$ and arbitrary $h\in \cH$ one has the
equality
\[
0=(g',\Phi h)-(g,\Psi h)=(\Phi^*g'-\Psi^*g,h),
\]
which implies $\Psi^*g-\Phi^*g'=0$.

(ii), (iii) If $T$ is symmetric then for $\{\Phi h,\Psi h\}\in T$
one obtains
\[
0=-i[(\Psi h,\Phi h)-(\Phi h,\Psi h)] =-i( (\Phi^*\Psi-
\Psi^*\Phi)h,h), \quad h\in\cH,
\]
and conversely. Similarly, one obtains the conditions~\eqref{Acor1}
for dissipative and accumulative linear relations.

(iv) It follows from\eqref{tau00} that
\begin{equation}
\label{Tinv} T-{z}=\{\,\{\Phi h,(\Psi-{z}\Phi) h\}:\,
h\in\cL \,\},
\end{equation}
Assume that ${z}\in\rho(T)$ and $(\Psi-{z}\Phi) h=0$. Then
$\Psi h=\Phi h=0$ and, hence, $h=0$ (by the assumption $\ker
\Phi\cap\ker\Psi=\{0\}$). Since $\ran
(\Psi-{z}\Phi)=\ran(T-{z})=\cH$, it follows
$0\in\rho(\Psi-{z}\Phi)$. Similarly, if
$0\in\rho(\Psi-{z}\Phi)$ one obtains from~\eqref{Tinv} that
${z}\in\rho(T)$ and
$(T-{z})^{-1}=\Phi(\Psi-{z}\Phi)^{-1}$.

(v), (vi) are immediate from (ii), (iii) and (iv).
\end{proof}

Now let a family of linear relations $\cF(\cdot)$ be represented in
the form
\begin{equation}
\label{tau1} \cF({z})=\{\Phi({z}),\Psi({z})\}:=
 \{\,\{\Phi({z})h,\Psi({z})h\}:\, h\in\cH\,\},
\end{equation}
where $\Phi(\cdot)$, $\Psi(\cdot)$ is a pair of holomorphic operator
functions on $\dC_+\cup\dC_-$. In the case when $\cF(\cdot)$ is a
Nevanlinna family the corresponding pair
$\{\Phi(\cdot),\Psi(\cdot)\}$ in the representation~\eqref{tau1} is
called the Nevanlinna pair.

\begin{definition}
\label{npair} A pair $\{\Phi,\Psi\}$ of $\cB({\cH})$-valued functions
$\Phi(\cdot)$, $\Psi(\cdot)$ holomorphic on $\cmr$ is said to be a
Nevanlinna pair if:
\begin{enumerate}
\item[(NP1)]
$\IM \Phi({z})^*\Psi({z})/\IM {z} \geq 0$, ${z}\in
\dC_+\cup\dC_-$;
\item[(NP2)]
$\Psi(\bar {z})^*\Phi({z})-\Phi(\bar
{z})^*\Psi({z})=0$, ${z} \in \cmr$;
\item[(NP3)]
$0\in\rho(\Psi({z})\pm i\Phi({z}))$, ${z} \in\dC_\pm$.
\end{enumerate}
\end{definition}


Two Nevanlinna pairs $\{\Phi(\cdot),\Psi(\cdot)\}$ and
$\{\Phi_1(\cdot),\Psi_1(\cdot)\}$ are said to be
\textit{equivalent}, if they generate the same graph $\cF(z)$ in
\eqref{tau1} for every $z\in\cmr$. In fact, the formula~\eqref{tau1}
establishes a one-to-one correspondence $\{\Phi,\Psi\}\mapsto \cF$
between the equivalence classes of Nevanlinna pairs and Nevanlinna
families $\cF(\cdot)\in \wt R(\cH)$; cf.
\cite[Proposition~2.4]{DHMS1}.

\begin{proposition}[\cite{DHMS1}]\label{propnp}
Let $\{\Phi,\Psi\}$ be a Nevanlinna pair of $\cB(\cH)$-valued functions
on $\dC_+\cup\dC_-$, and let $\cF(\cdot)$ be defined by
\eqref{tau1}. Then $\cF(\cdot)$ is a Nevanlinna family.

Conversely, if $\cF(\cdot)\in \wt R(\cH)$ then there exists a
Nevanlinna pair $\{\Phi,\Psi\}$ of $\cB(\cH)$-valued functions on
$\dC_+\cup\dC_-$ such that \eqref{tau1} holds.
\end{proposition}


\begin{proof}
Let $\{\Phi,\Psi\}$ be a Nevanlinna pair. Then it follows from
(NP1), (NP3) and Proposition~\ref{CRIT} that the linear relation
$\cF({z})$ is maximal dissipative (maximal accumulative) for all
${z}\in\dC_+$ (${z}\in\dC_-$). The assumption (NP2) concerning
$\{\Phi,\Psi\}$ means that $\cF(\bar{z})\subset \cF({z})^*$.
According to \eqref{tau01}
\begin{equation}
\label{Mstar} \cF({z})^* =\{\,\{h,h'\}\in\cH^2:\,
\Psi({z})^*h-\Phi({z})^*h'=0 \,\},
\end{equation}
and, hence
\[
 \cF({z})^*\pm i=\left\{ \left\{h,g \right\}:
 (\Psi({z})^*\pm i\Phi({z})^*)h=\Phi({z})^*g \right\}.
\]
Using (NP3) one obtains $\ker(\cF({z})^*\pm i)=\{0\}$ and
$\ran(\cF({z})^*\pm i)=\cH$ for ${z}\in\dC_\mp$. Similarly
$\ker(\cF(\bar{z})\pm i)=\{0\}$ and $\ran(\cF(\bar{z})\pm i)=\cH$,
${z}\in\dC_\mp$. Hence, $(\cF({z})^*\pm i)^{-1}$ and $(\cF({\bar z})
\pm i)^{-1}$, ${z}\in\dC_\mp$, both are everywhere defined operators
and, thus, the inclusion $\cF(\bar{z})\subset \cF({z})^*$ must hold
as an equality $\cF(\bar{z})= \cF({z})^*$, ${z}\in\cmr$. This proves
(NF2) and (NF3) with $w=\pm i$.

Conversely, assume that $\cF(\cdot)\in \wt R(\cH)$. Define
$\Phi(\cdot)$ and $\Psi(\cdot)$ by
\[
 \Phi({z})=(\cF({z})\pm i)^{-1},
\quad
 \Psi({z})=I\mp i(\cF({z})\pm i)^{-1},
\quad {z}\in \dC_\pm.
\]
Then $\cF(\cdot)$ has the representation \eqref{tau1}. The property
(NF3) implies that $\Phi(\cdot)$, $\Psi(\cdot)$ are holomorphic on
$\dC_+\cup\dC_-$ with the values in $\cB(\cH)$. Clearly,
$\Psi({z})\pm i\Phi({z})=I$ and hence (NP3) holds. Moreover,
the symmetry condition (NP2) is obvious and the positivity condition
(NP1) follows from (NF1) in view of
\[
  \frac{((\Phi({z})^*\Psi({z})
           -\Psi({z})^*\Phi({z}))h,h)}{\IM {z}}
  =\frac{\IM (\Psi({z})h,\Phi({z})h)}{\IM {z}}
  \geq 0.
\]
This completes the proof.
\end{proof}

Let $\{\Phi,\Psi\}$ be a Nevanlinna pair and let $\cF(\cdot)$ be a
family of linear relations associated with the Nevanlinna pair
$\{\Phi,\Psi\}$. The Cayley transform $\cC({z})$ of $\cF({z})$ is
given by
\begin{equation}
\label{Cayley}
  \cC({z})
  =(\Psi({z})-i\Phi({z}))(\Psi({z})+i\Phi({z}))^{-1}\quad({z}\in\dC_+).
\end{equation}
The operator-valued function $\cC({z})$ belongs to the Schur class
$\cS({\mathcal H})$, i.e. $\cC({z})$ is holomorphic on $\dC_+$ and
takes values in the set of contractive operators on ${\mathcal H}$
for all ${z}\in\dC_+$. For every operator-valued function
$\cC\in\cS({\mathcal H})$ the kernel
\begin{equation}\label{eq:Ker_K}
    {\sf K}({z},\,{w})=\frac{I-\cC({w})^*\cC({z})}{-i({z}-\bar{{w}})},\quad z,w\in\dC_+.
\end{equation}
is nonnegative on $\dC_+$ in a sense that for every choice of $z_j\in\dC_+$ and $h_j\in\cH$ $(j=1,\dots,n)$ the quadratic form
\[
\sum_{j=1}^n\left({\sf K}({z_i},\,{z_j})h_i,h_j\right)\xi_i\bar\xi_j
\]
is nonnegative.
\begin{proposition}
\label{Nkernel} Let $\{\Phi,\Psi\}$ be a Nevanlinna pair. Then the
following kernel is nonnegative on $\dC_+$:
\[
 {\sf N}_{\Phi,\Psi} ({z},{w})
  =\frac{\Phi({w})^* \Psi({z})-
   \Psi({w})^* \Phi({z})}{{z}-\bar {w}},
\quad
 {z},{w} \in \dC_+.
\]
\end{proposition}
\begin{proof}
Let $\cF(\cdot)$ be a family of linear relations associated with the
Nevanlinna pair $\{\Phi,\Psi\}$ and let $\cC({z})$ be the Cayley
transform of $\cF({z})$ given by~\eqref{Cayley}. It follows from the
equality
\begin{equation}\label{eq:KerK_N}
   {\sf K}({z},\,{w})
=2(\Psi({w})+i\Phi({w}))^{-*}{\sf N}_{\Phi,\Psi}({z},{w})
      (\Psi({z})+i\Phi({z}))^{-1}
\end{equation}
that the kernel ${\sf N}_{\Phi,\Psi}({z},{w})$ is nonnegative on
$\dC_+$; cf. \cite{SNF}.
\end{proof}


Nonnegativity of the kernel $ {\sf N}_{\Phi,\Psi}$ implies the
following properties for Nevanlinna families and reflect maximum
principle in the class $\wt R(\cH)$.

\begin{proposition}\label{NFmaxmod}
Let $\{\Phi(\cdot), \Psi(\cdot)\}$ be a Nevanlinna pair and let
$\cF(\cdot)\in\wt R(\cH)$ be the corresponding Nevanlinna family.
Let $z_0\in\cmr$ be fixed and let $S$ be a symmetric relation in
$\cH$. Then
\begin{equation}\label{Sinvar}
 S\subset \cF(z_0) \quad \Rightarrow \quad S\subset \cF(z) \quad
 \text{for every } z\in\cmr.
\end{equation}
In particular,
\begin{equation}\label{eq:F_capF*0}
    \cF({z})\cap \cF(\bar{z})\equiv \cF({z_0})\cap \cF(\bar{z_0}) \quad ({z}\in \cmr)
\end{equation}
is a maximal symmetric subspace $S$ satisfying the inclusion
$S\subset \cF(z)$ for some (equivalently for every) $z\in\cmr$.
Moreover,
\begin{equation}\label{eq:F_cap_F*}
    \cF({z})\cap \cF(\bar{z})=\{\,\{\Phi(z)u,\Psi(z)u\}:\,
    u\in \ker({\sf N}_{\Phi,\Psi}({z},{z}))\,\} \quad (z\in\dC_+).
\end{equation}
\end{proposition}

\begin{proof}
Assume that $S\subset\cF(z_0)$ and let
$\{\Phi(z_0)h_0,\Psi(z_0)h_0\}\in S$ and $\{\Phi(z)h,\Psi(z)h\}\in
\cF(z)$ with $z\neq \bar{z_0}$. By Proposition \ref{Nkernel} the
matrix
\[
\begin{pmatrix}
({\sf N}_{\Phi,\Psi}({z}_0,{z}_0)h_0,h_0) &
({\sf N}_{\Phi,\Psi}({z}_0,{z})h_0,h) \\
({\sf N}_{\Phi,\Psi}({z},{z}_0)h,h_0) & ({\sf
N}_{\Phi,\Psi}({z},{z})h,h)
\end{pmatrix} \quad(h\in \cH,\,\, h_0\in\cH_0)
\]
is nonnegative, and since $S$ is symmetric the left-upper corner
equals to 0. This implies that $({\sf
N}_{\Phi,\Psi}({z}_0,{z})h_0,h)=0$ for all $h\in \cH$ or,
equivalently,
\[
(\Psi(z_0)h_0,\Phi(z)h)_\cH=(\Phi(z_0)h_0,\Psi(z)h)_\cH
\]
for all $h\in \cH$. Therefore, $\{\Phi(z_0)h_0,\Psi(z_0)h_0\}\in
\cF(z)^*=\cF(\bar{z})$ for all $z\neq \bar{z_0}$. This proves
\eqref{Sinvar}.

Since $\cF({z})\cap \cF(\bar{z})=\cF({z})\cap\cF(z)^*$, this
subspace is symmetric and hence its maximality as a symmetric subset
of $\cF(z)$ follows from \eqref{Sinvar}. Moreover, the invariance
property \eqref{eq:F_capF*0} is also immediate from \eqref{Sinvar}.

Finally, the equality \eqref{eq:F_cap_F*} follows from the general
formula $F\cap F^*=F\uphar\ker(\IM F)$ for a linear relation $F$
with $\mul F=\mul F^*$; see \cite[Section 5.1]{HdSSz2009}. In fact,
here in view of \eqref{tau01} $\{\Phi(z)u,\Psi(z)u\}\in \cF(z)^*$
precisely, when
\[
 \Psi(z)^*\Phi(z)u -\Phi(z)^*\Psi(z)u=0,
\]
or, equivalently, ${\sf N}_{\Phi,\Psi}({z},{z})u=0$.
\end{proof}


\begin{corollary}\label{NFCor}
Let $\{\Phi(\cdot), \Psi(\cdot)\}$ and $\cF(\cdot)\in\wt R(\cH)$ be
as in Proposition \ref{NFmaxmod} and let $z_0\in\cmr$ and $a\in\dR$.
Then the following statements hold for all $z\in\cmr$:
\begin{enumerate}
\def\labelenumi{\rm (\roman{enumi})}
 \item $\ker (\cF(z)-a)=\ker (\cF(z_0)-a)$;
 \item $\ker (\cF(z)-\cF(\bar{z}))=\{f\in \cH:\, \{f,g\}\in \cF({z_0})\cap \cF(\bar{z_0})\}$;
 \item $\ker (\cF(z)^{-1}-\cF(\bar{z})^{-1})=\{g\in \cH:\, \{f,g\}\in \cF({z_0})\cap
 \cF(\bar{z_0})\}$;
 \item $\mul \cF(z)=\mul \cF(z_0)$.
\end{enumerate}
\end{corollary}
\begin{proof}
The statements (i) and (iv) follow from Proposition \ref{NFmaxmod},
since $\ker (\cF(z)-a)$ with $a\in\dR$ and $\mul \cF(z)$ are
symmetric subspaces of $\cF(z)$.

On the other hand, the formulas (ii) and (iii) clearly hold when
$z=z_0$. The independence from $z\in\cmr$ follows from
\eqref{eq:F_capF*0}.
\end{proof}

The statements of the following lemma are based on the maximum principle for the class $\cS({\mathcal H})$ and, apparently, are well-known.
\begin{lemma}\label{lem:MaxP}
Let $\cC(\cdot)\in\cS({\mathcal H})$, ${z}_0\in\dC_+$, $\alpha\in\dC$ and $|\alpha|=1$.
Then the following statements hold:
\begin{enumerate}
    \item[(1)] If $0\in\sigma_p({{\sf K}({z}_0,\,{z}_0)})$, then $0\in\sigma_p({{\sf K}({z},\,{z})})$
for all ${z}\in\dC_+$ and in this case
\begin{equation*}
\ker {\sf K}({z}_0,\,{z}_0)=\ker {\sf K}({z},\,{z});
\end{equation*}
    \item[(2)] If $0\in\rho({{\sf K}({z}_0,\,{z}_0)})$, then $0\in\rho({{\sf K}({z},\,{z})})$ for all
${z}\in\dC_+$.
 \item[(3)] If $0\in\sigma_c({{\sf K}({z}_0,\,{z}_0)})$, then $0\in\sigma_c({{\sf K}({z},\,{z})})$ for all
${z}\in\dC_+$.
\item[(4)] If $\alpha\in\sigma_p({\cC({z}_0)})$, then
$\alpha\in\sigma_p({\cC({z})})$ for all ${z}\in\dC_+$ and in this case
\begin{equation}\label{eq:KerId0}
\ker (\cC({z}_0)-\alpha)=\ker (\cC({z})-\alpha);
\end{equation}
\item[(5)] If $\alpha\in\rho({\cC({z}_0)})$, then
$\alpha\in\rho({\cC({z})})$ for all ${z}\in\dC_+$;
\item[(6)] If $\alpha\in\sigma_c({\cC({z}_0)})$, then
$\alpha\in\sigma_c({\cC({z})})$ for all ${z}\in\dC_+$.
\end{enumerate}
\end{lemma}
\begin{proof}
(1) Let $0\in\sigma_p({{\sf K}({z}_0,\,{z}_0)})$  and let ${\sf
K}({z}_0,\,{z}_0)h_0=0$  for some ${z}_0\in\dC_+$ and $h_0\ne 0$.
Then the following matrix
\[
\begin{pmatrix}
({\sf K}({z}_0,{z}_0)h_0,h_0) &
({\sf K}({z}_0,{z})h_0,h) \\
({\sf K}({z},{z}_0)h,h_0) & ({\sf K}({z},{z})h,h)
\end{pmatrix},
\]
is nonnegative for all ${z}\in\dC_+$, $h\in{\mathcal H}$ and hence $({\sf K}({z}_0,{z})h_0,h)=0$ for all
 $h\in{\mathcal H}$. This implies that the contraction $T=\cC({z})^*\cC({z}_0)$ has an eigenvector $h_0$ corresponding to the eigenvalue 1: $Th_0=h_0$. Therefore, $h_0$ is also an eigenvector for the operator $T^*=\cC({z}_0)^*C({z})$ ($T^*h_0=h_0$) and, hence,
\[
\|h_0\|=\|\cC({z}_0)^*\cC({z})h_0\|\le\|\cC({z})h_0\|.
\]
This implies the inequality
\[
(0\le)({\sf K}({z},{z})h_0,h_0)=(h_0,h_0)-(\cC({z})h_0,\cC({z})h_0)\le 0,
\]
which means that ${\sf K}({z},{z})h_0=0$. This proves the statement (1).

(2) Let $0\in\rho({{\sf K}({z}_0,\,{z}_0)})$  for some
${z}_0\in\dC_+$. Then the operator $\cC({z}_0)$ is a strict contraction and the space ${\mathcal H}^2$ admits the decomposition
\begin{equation}\label{eq:decom}
    {\mathcal H}^2=\ran\left(%
\begin{array}{c}
  I \\
  \cC({z}_0) \\
\end{array}%
\right) \dotplus
\ran\left(%
\begin{array}{c}
  \cC({z}_0)^* \\
  I \\
\end{array}%
\right).
\end{equation}
Assume that $0\in\sigma_c({{\sf K}({z},\,{z})})$. Then there exists a sequence $h_n\in{\mathcal H}$, such that $\|h_n\|=1$ and
${\sf K}({z},\,{z})h_n\to 0$ as $n\to\infty$. Using the decomposition~\eqref{eq:decom}, one obtains
\begin{equation}
\label{eq:decom1}
\left(%
\begin{array}{c}
  h_n \\
 \cC({z})h_n \\
\end{array}%
\right)=
\left(%
\begin{array}{c}
  h_n' \\
 \cC({z}_0)h_n' \\
\end{array}%
\right) \dotplus
\ran\left(%
\begin{array}{c}
  \cC({z}_0)^*h_n'' \\
  h_n'' \\
\end{array}%
\right),
\end{equation}
where $h_n',h_n''\in{\mathcal H}$. Since the matrix
\[
\begin{pmatrix}
({\sf K}({z}_0,{z}_0)h_n',h_n') &
({\sf K}({z}_0,{z})h_n',h_n) \\
({\sf K}({z},{z}_0)h_n,h_n') & ({\sf
K}({z},{z})h_n,h_n)
\end{pmatrix},
\]
is nonnegative, then
\[
({\sf K}({z}_0,{z})h_n',h_n)\to 0
\]
as $n\to\infty$. Using~\eqref{eq:decom1} one obtains
\[
((I-\cC({z}_0)^*\cC({z}_0))h_n',h_n')\to 0.
\]
By the assumption $0\in\rho({{\sf K}({z}_0,\,{z}_0)})$ this implies $h_n'\to 0$. Next, the equality
\[
((I-\cC({z})^*\cC({z}))h_n,h_n)=((I-\cC({z}_0)^*\cC({z}_0))h_n',h_n')
-((I-\cC({z}_0)\cC({z}_0)^*)h_n'',h_n'')
\]
yields
\[
((I-\cC({z}_0)\cC({z}_0)^*)h_n'',h_n'')\to 0,
\]
which, in view of the condition $0\in\rho(I-\cC({z}_0)^*\cC({z}_0))$
implies that $h_n''\to 0$. Therefore, $h_n\to 0$ as $n\to\infty$ and
this contradicts the equalities $\|h_n\|=1$.

(3)  If $0\in\sigma_c({{\sf K}({z}_0,\,{z}_0)})$, then by (1) and (2) $0\not\in\sigma_p({{\sf K}({z},\,{z})})\cup\rho({{\sf K}({z},\,{z})})$ and hence
 $0\in\sigma_c({{\sf K}({z},\,{z})})$.

(4) Let $\alpha\in\sigma_p({\cC({z}_0)}$ $(|\alpha|=1)$ and
$\cC({z}_0)h_0=\alpha h_0$ for some vector $h_0\ne 0$. Then $h_0$ is
an eigenvector for the contraction $T=\cC({z})^*$, corresponding to
the eigenvalue $\alpha^{-1}$ and  for the contraction
$T^*=\cC({z})$, corresponding to the eigenvalue
$\alpha^{-*}=\alpha$. This proves the equality~\eqref{eq:KerId0}.


(5) Let $\alpha\in\rho(\cC({z}_0))$ $(|\alpha|=1)$. Then for every
${z}\in\dC_+$ and $u\in\cH$ the harmonic function
\[
h_u({z}):=\RE\{
e^{-i\arg(\alpha)}\left((\alpha-\cC({z}))u,u\right)\}\ge 0
\]
is nonnegative and satisfies the inequality $h_u({z}_0)\ge
q\|u\|^2>0$ for some $q\in(0,1)$. By Harnack's inequality (cf.
Section 4 below) for every ${z}\in\dC_+$ there are positive
constants $c_1({z})$ and $c_2({z})$, such that
\[
c_1({z})h_u({z}_0)\le h_u({z})\le c_2({z})h_u({z}_0).
\]
It is emphasized that constants $c_1(z)$ and $c_2(z)$ do not depend on
$u\in \cH$.
Therefore,
\[
h_u({z})\ge qc_1({z})\|u\|^2>0\quad\mbox{for all}\quad u\in\cH
\] and hence $\alpha\in\rho(\cC({z}))$.

(6) Let $\alpha\in\sigma_c({\cC({z}_0)})$
$(|\alpha|=1)$. Then by (4) and (5) $\alpha\not\in\sigma_p({{\cC}({z})})\cup\rho({{\cC}({z})})$. Moreover,
$\alpha\not\in\sigma_r({{\cC}({z})})$, since otherwise we would have $\bar\alpha\in\sigma_p({{\cC}({z})^*})$ and hence
$\bar\alpha\in\sigma_p({{\cC}({z}_0)^*})$ which contradicts the assumption $\alpha\in\sigma_c({\cC({z}_0)})$. This completes the proof of (6).
\end{proof}
In order to adapt the above statements to the class $\wt R(\cH)$ we will need the following lemma connecting the spectral properties of a Nevanlinna family $\cF(\cdot)\in\wt R(\cH)$ with the spectral properties of its Cayley transform $\cC(z)$.
\begin{lemma}
\label{lem:F_C} Let $\{\Phi(\cdot), \Psi(\cdot)\}$ be a Nevanlinna
pair, let $\cF(\cdot)\in\wt R(\cH)$ be the corresponding
Nevanlinna family and let the operator function $\cC(z)$ and the kernel ${\sf K}({z},\,{w})$ be  defined by~\eqref{Cayley} and~\eqref{eq:Ker_K}. Let $a\in \dR$,  $\alpha=(a-i)(a+i)^{-1}$ and ${z}\in \dC_+$. Then the following equivalences hold:
\begin{enumerate}
\def\labelenumi{\rm (\roman{enumi})}
\item $0\in \sigma_p({\sf N}_{\Phi,\Psi} ({z},{z}))\Longleftrightarrow 0\in \sigma_p({\sf K}({z},{z}))$;
\item $0\in \sigma_c({\sf N}_{\Phi,\Psi} ({z},{z}))\Longleftrightarrow 0\in \sigma_c({\sf K}({z},{z}))$;
\item $0\in \rho({\sf N}_{\Phi,\Psi} ({z},{z}))\Longleftrightarrow 0\in \rho({\sf K} ({z},{z}))$;
\item $a\in \sigma_p(\cF({z}))\Longleftrightarrow \alpha\in \sigma_p(\cC({z}))$;
\item $a\in \sigma_c(\cF({z}))\Longleftrightarrow \alpha\in \sigma_c(\cC({z}))$;
\item $a\in \rho(\cF({z}))\Longleftrightarrow \alpha\in \rho(\cC({z}))$;
\item $\cF({z})\in \cB(\cH)\Longleftrightarrow 1\in \rho(\cC(z))$.
\end{enumerate}
\end{lemma}
\begin{proof}
The equivalences (i)-(iii) are implied by the identity~\eqref{eq:KerK_N}.

Notice, that $a\in \sigma_p(\cF({z}))$ if and only if
\begin{equation}\label{eq:PsiPhi}
    (\Psi(z)-a\Phi(z))u=0\quad\mbox{for some }\quad u\in\cH\setminus\{0\}.
\end{equation}
If \eqref{eq:PsiPhi} holds then by (NP3)
$h:=(\Psi(z)+i\Phi(z))u=(a+i)\Phi(z)u\ne 0$ and in view
of~\eqref{Cayley}
\[
 \cC(z)h=(a-i)\Phi(z)u=\frac{a-i}{a+i}h=\alpha h.
\]
Therefore, $\alpha\in \sigma_p(\cC({z}))$.
Conversely, if $\alpha\in \sigma_p(\cC({z}))$, and $\cC(z)h=\alpha h$ for some $h\in\cH\setminus\{0\}$, then~\eqref{eq:PsiPhi} holds for $u=(\Phi(z)+i\Phi(z))^{-1}h(\ne 0)$
and hence $a\in \sigma_p(\cF({z}))$. This proves (iv).

The equivalences (v)-(vi) follows from the equality~\eqref{Cayley} and the equivalences
\[
a\in \sigma_c(\cF({z}))\Longleftrightarrow \ker(\Psi(z)-a\Phi(z))=\{0\},\,\mbox{ and }\,\ran(\Psi(z)-a\Phi(z))\,\mbox{ is dense in }\, \cH;
\]
\[
a\in \rho(\cF({z}))\Longleftrightarrow 0\in\rho(\Psi(z)-a\Phi(z)).
\]

Similarly, (vii) follows from the equality~\eqref{Cayley} and the equivalence
\[
\cF({z})\in \cB(\cH)\Longleftrightarrow 0\in\rho(\Phi(z)). \qedhere
\]
\end{proof}
\begin{theorem}
\label{invar} Let $\{\Phi(\cdot), \Psi(\cdot)\}$ be a Nevanlinna
pair and let $\cF(\cdot)\in\wt R(\cH)$ be the corresponding
Nevanlinna family. Let ${z}_0\in \dC_+$ and let $a \in \dR$. Then
the following statements hold:
\begin{enumerate}
\def\labelenumi{\rm (\roman{enumi})}

\item if $0\in \sigma_p({\sf N}_{\Phi,\Psi} ({z}_0,{z}_0))$,
      then $0\in \sigma_p({\sf N}_{\Phi,\Psi}({z},{z}))$
for all ${z}\in \cmr$;
\item if $0\in \sigma_c({\sf N}_{\Phi,\Psi} ({z}_0,{z}_0))$,
      then $0\in \sigma_c({\sf N}_{\Phi,\Psi}({z},{z}))$
for all ${z}\in \cmr$;

\item if $0\in \rho({\sf N}_{\Phi,\Psi} ({z}_0,{z}_0))$,
then
      $0\in \rho({\sf N}_{\Phi,\Psi} ({z},{z}))$ for all ${z}\in \cmr$;

\item if $a\in \sigma_p(\cF({z}_0))$, then
      $a\in \sigma_p(\cF({z}))$ for all ${z}\in
\cmr$ and in this case
\[
\ker (\cF({z})-a)=\ker (\cF({z_0})-a);
\]

\item if $a\in \sigma_c(\cF({z}_0))$, then
      $a\in \sigma_c(\cF({z}))$ for all ${z}\in
\cmr$;

\item if $a\in \rho(\cF({z}_0))$, then
      $a\in \rho(\cF({z}))$ for all ${z}\in \cmr$;

\item if $\cF({z}_0)\in \cB(\cH)$, then
      $\cF({z})\in \cB(\cH)$ for all ${z}\in \cmr$;
\item $\mul(\cF({z}))$ does not depend on ${z}\in \cmr$.
\end{enumerate}
\end{theorem}

\begin{proof}
Statements (i), (iv) and (viii) have been derived already in
Corollary~\ref{NFCor}.

Statements  (ii) and (iii) follow from Lemma~\ref{lem:MaxP} (2)-(3)
and Lemma~\ref{lem:F_C} (ii)--(iii).

Statements (iv) -- (vii), follow from Lemma~\ref{lem:MaxP} (4)-(6)
and Lemma~\ref{lem:F_C} (iv) -- (vii).
%
\end{proof}


\begin{remark}\label{RemInvariant}
The invariance results in Theorem~\ref{invar} can be obtained from
the realization of a Nevanlinna family as a Weyl family of a
boundary relation. Such an approach for proving these facts was used
in \cite{DHMS06}; see, in particular, \cite[Lemma 4.1, Prop.
4.18]{DHMS06}. Also other models giving realizations for Nevanlinna
families can be used in establishing such invariance results; we
mention, in particular, the functional models which can be found
from \cite{BDHdS2011,BHdS2009}.
\end{remark}

\begin{proposition}
\label{rwi1} Let $\{\Phi, \Psi\}$ be a Nevanlinna pair and let
$\cF(\cdot)\in\wt R(\cH)$ be the corresponding Nevanlinna family.
Let ${z}\in \dC_+$. Then:
\begin{enumerate}
\def\labelenumi{\rm (\roman{enumi})}
\item
 $\cF(\cdot)\in R^s(\cH)$ if and only if $0\not\in\sigma_p({\sf
N}_{\Phi,\Psi}({z},{z}))$;

\item
 $\cF(\cdot)\in R^u(\cH)$ if and only if $0\in\rho({\sf
 N}_{\Phi,\Psi}({z},{z}))$.
\end{enumerate}
\end{proposition}
\begin{proof}
(i) Let  $h\in\ker{\sf N}_{\Phi,\Psi}({z},{z})$, that is
$(\Phi({z})^* \Psi({z})-\Psi({z})^* \Phi({z}))h=0$. Then it follows
from~\eqref{Mstar} that
\[
\{\Phi({z})h,
\Psi({z})h\}\in \cF({z})\cap \cF({z})^*
\]
and, therefore,
$h=0$ if and only if $\cF({z})\cap \cF({z})^*=\{0\}$.

(ii) Let $f$, $f'\in\cH$ and let $h$, $g$ satisfy the equations
\begin{equation}
\label{DHMS06} \Phi({z})h+\Phi(\bar{z})g=f, \quad
\Psi({z})h+\Psi(\bar{z})g=f',
\end{equation}
Then it follows from (NP2) that
\begin{equation}
\label{TRANS2} {\sf
N}_{\Phi,\Psi}({z},{z})h=\Psi({z})^*f-\Phi({z})^*f',
\quad {\sf N}_{\Phi,\Psi}(\bar{z},\bar{z})g
=\Psi(\bar{z})^*f-\Phi(\bar{z})^*f'.
\end{equation}
Assume that $0\in\rho({\sf N}_{\Phi,\Psi}({z},{z}))$. Then it
follows from~\eqref{TRANS2} and Theorem~\ref{invar} that the
system~\eqref{DHMS06} has a unique solution for all $f$, $f'\in\cH$
and, therefore, $\cF({z})\hplus \cF({z})^*=\cH^2$.

Conversely, let $\cF(\cdot)\in R^u(\cH)$ and thus the
system~\eqref{DHMS06} has a unique solution for all $f$, $f'\in\cH$.
Then it follows from the first equation in~\eqref{TRANS2} and the
hypothesis (NP3) that $\ran {\sf N}_{\Phi,\Psi}({z},{z})=\cH$. This
implies that $0\in\rho({\sf N}_{\Phi,\Psi}({z},{z}))$.
\end{proof}

\begin{proposition}
\label{Runif+} Let $ \cF(\cdot)\in R^u(\cH)$. Then $
\cF({z})\in\cB(\cH)$ and $\cF({z})^{-1}\in\cB(\cH)$ for every
${z}\in\cmr$. In particular, the following equality holds
$R^u(\cH)=R^u[\cH]$.
\end{proposition}

\begin{proof}
Let $\{\Phi,\Psi\}$ be a Nevanlinna pair associated to $ M$. It is
enough to prove that $\Phi({z})$ and $\Psi({z})$ are
invertible. Now assume, for instance, that $\Phi({z})h_n\to 0$
for some sequence $h_n\in\cH$, $\|h_n\|=1$. This together with
$0\in\rho({\sf N}_{\Phi,\Psi}({z},{z}))$ shows that for some
$\alpha>0$ one has
\[
  \alpha \le({\sf N}_{\Phi,\Psi}({z},{z})h_n,h_n)_\cH\to 0,
\]
a contradiction. Since $\ran \Phi({z})=\dom \cF({z})$
($\ran \Psi({z})=\ran \cF({z})$) is dense in $\cH$, one
concludes that $\Phi({z})$ must be invertible. A similar
argument shows that $\Psi({z})$ is invertible.
\end{proof}

We finish this section with some further properties of Nevanlinna
pairs. The next statement can be found e.g. from
\cite[Proposition~2.4]{DHMS1}.


\begin{lemma}
Two Nevanlinna pairs $\{\Phi,\Psi\}$ and $\{\Phi_1,\Psi_1\}$ are
equivalent if and only if $\Phi_1 ({z})=\Phi({z})\chi({z})$ and
$\Psi_1({z})=\Psi({z})\chi({z})$ for some operator function
$\chi(\cdot)\in\cB(\cH)$ which is holomorphic and invertible on
$\dC_+\cup\dC_-$.
\end{lemma}
\begin{proof}
By definition the Nevanlinna pairs $\{\Phi,\Psi\}$ and
$\{\Phi_1,\Psi_1\}$ are equivalent if and only if the ranges of the
block operators
\[
T({z})=\begin{pmatrix}\Phi({z})\\
\Psi({z})\end{pmatrix}, \quad \wt
T_1({z})=\begin{pmatrix}\Phi_1({z})\\ \Psi_1({z})\end{pmatrix}
\]
coincide with the graph of the corresponding Nevanlinna family
$\cF(z)$, $z\in\cmr$. It is well known (from Douglas' lemma) that
the equality $\ran T({z})=\ran T_1({z})$ implies the existence of a
bounded operator $\chi(z)\in\cB(\cH)$ such that
$T({z})=T_1({z})\chi(z)$. Thus, $\Phi({z})=\Phi_1({z})\chi(z)$ and
$\Psi({z})=\Psi_1({z})\chi(z)$, and hence
\[
 \Phi({z})\pm i\Psi({z})=(\Phi_1({z})\pm i\Psi_1({z}))\chi(z), \quad
 z\in\cmr.
\]
In view of (NP3) this implies that $\chi(z)$ is bounded with bounded
inverse and holomorphic in $z\in\dC_\pm$.
\end{proof}

As follows from Proposition~\ref{NPpr} the conditions (NP3) and
(NF3) can be replaced, for instance, by
\[
  0\in \rho(\Psi({z})+{w}\Phi({z}))\quad\mbox{and}\quad
  0\in \rho(\cF({z})+{w} I),
\]
respectively, for some (equivalently for every) ${w} \in \dC_\pm$
and for all ${z}$ in the same halfplane as ${w}$. Moreover, the
following more general statement holds.

\begin{proposition}
\label{NPpr} Let $\{\Phi(\cdot),\Psi(\cdot)\}$ be a Nevanlinna pair,
let $W$ be a unitary operator in the Kre\u{\i}n space
$(\cH^2,J_\cH)$, and let
\[
\begin{pmatrix}\wt\Phi({z})\\
\wt\Psi({z})\end{pmatrix} =W\begin{pmatrix}\Phi({z})\\
\Psi({z})\end{pmatrix}.
\]
Then $\{\wt\Phi(\cdot),\wt\Psi(\cdot)\}$ is also a Nevanlinna pair.
In particular, if $X=X^*\in\cB(\cH)$, $Y$ is an invertible operator in
$\cB(\cH)$ and $M(\cdot)\in R^u[\cH]$, each of the following pairs is
also a Nevanlinna pair:
\begin{equation}
\label{NPs} \{\Phi(\cdot),\Psi(\cdot)+X\Phi(\cdot)\},\quad
\{Y^{-1}\Phi(\cdot),Y^{*}\Psi(\cdot)\},\quad
\{-\Psi(\cdot),\Phi(\cdot)\},\quad
\{\Phi(\cdot),\Psi(\cdot)+M(\cdot)\Phi(\cdot)\}.
\end{equation}
\end{proposition}

\begin{proof}
Consider $\cF({z})$ and $\wt\cF({z})$ as the ranges of
the block operators
\[
T({z})=\begin{pmatrix}\Phi({z})\\
\Psi({z})\end{pmatrix}, \quad \wt
T({z})=\begin{pmatrix}\wt\Phi({z})\\
\wt\Psi({z})\end{pmatrix}.
\]
Then the kernel ${\sf N}_{\Phi, \Psi} ({z},{w})$ can be
represented as follows:
\begin{equation}
\label{Nkern2}
 {\sf N}_{\Phi, \Psi} ({z},{w})
  =\frac{T({w})^*J_\cH T({z})}{-i({z}-\bar{{w}})}.
\end{equation}
The properties (NP1), (NP2) for $\{\wt\Phi(\cdot),\wt\Psi(\cdot)\}$
are implied by the equalities
\begin{equation}
\label{NN2}
 {\sf N}_{\wt\Phi, \wt\Psi} ({z},{w})
  =\frac{\wt T({w})^*J_\cH\wt T({z})}{-i({z}-\bar{{w}})}
  =\frac{ T({w})^*J_\cH T({z})}{-i({z}-\bar{{w}})}
  ={\sf N}_{\Phi, \Psi} ({z},{w}).
\end{equation}
The graph $\cF({z})$ can be treated as a a maximal nonnegative
subspace of the Kre\u{\i}n space $(\cH^2,J_\cH)$ for ${z}\in\dC_+$;
see \cite[Section~2]{DHMS06}. Since $\wt \cF({z})$ is the range of
$\wt T({z})$ it has the same property and, therefore, $\wt
\cF(\cdot)\in\wt R_\cH$. By Proposition~\ref{propnp}
$\{\wt\Phi,\wt\Psi\}$ is a Nevanlinna pair.

Applying this statement to the pair $\{\Phi,\Psi\}$ and the matrices
\[
W=\begin{pmatrix} I & 0\\X & I\end{pmatrix},\quad W=\begin{pmatrix}
Y^{-1} & 0\\0 & Y^*\end{pmatrix},\quad W=\begin{pmatrix} 0 & -I\\I &
0\end{pmatrix}
\]
one shows that the first three pair in~\eqref{NPs} are Nevanlinna
pairs. The properties (NP1), (NP2) for the pair
$\{\wt\Phi(\cdot),\wt\Psi(\cdot)\}
=\{\Phi(\cdot),\Psi(\cdot)+M(\cdot)\Phi(\cdot)\}$ are implied by the
identity
\[
 {\sf N}_{\wt\Phi, \wt\Psi} ({z},{w})
 = {\sf N}_{\wt\Phi, \wt\Psi} ({z},{w})
+\Phi({w})\frac{M({z})-M({w})^*}{{z}-\bar{w}}\Phi({z}).
\]
To show that the operator $\wt\Psi({z})+i \wt\Phi({z})$ is
invertible for some ${z}\in\dC_+$, set $X=\mbox{Re }M({z})$,
$Y=\mbox{Im }M({z})$ and apply the previous statement to the
pairs:
\[
\{\Phi_1({z}),\Psi_1({z})\}=
\{(Y+I)^{1/2}\Phi({z}),(Y+I)^{-1/2}\Psi({z})\},
\]
\[
\{\Phi_2({z}),\Psi_2({z})\}=
\{\Phi_1({z}),\Psi_1({z})+(Y+I)^{-1/2}X(Y+I)^{-1/2}\Phi_1({z})\}.
\]
Since these pairs are maximal dissipative it follows that the
operator
\[
\wt\Psi({z})+i \wt\Phi({z})=(Y+I)^{1/2}(\Psi_2({z})+i
\Phi_2({z}))
\]
is also invertible.
\end{proof}

\begin{remark}\label{later0}
The connection between Nevanlinna pairs
$\{\Phi(\cdot),\Psi(\cdot)\}$ and Nevalinna families $\cF(\cdot)\in
\wt R(\cH)$ in Proposition~\ref{propnp} implies some invariance
properties for the pair $\{\Phi(\cdot),\Psi(\cdot)\}$ via
Theorem~\ref{invar}. We indicate here a couple of the underlying
connections.
\begin{enumerate}
\def\labelenumi{\rm (\roman{enumi})}
\item If $\{\Phi(\cdot),\Psi(\cdot)\}$ corresponds to $\cF(\cdot)$ then the transformed
pair $\{\Psi(\cdot), -\Phi(\cdot)\}$ corresponds to the inverse
$-\cF(z)^{-1}$ and, moreover,
\[
 {\mathbf N}_{\Psi,-\Phi}(z,w)={\mathbf N}_{\Phi,\Psi}(z,w),\quad z,w\in\cmr.
\]
\item The kernels of $\Phi(z)$ and $\Psi(z)$ do not depend on
$z\in\cmr$; namely,
\[
 \mul \cF(z)=\mul (\cF(z)\pm iI)=\mul \{\Phi(z), \Psi(z)\pm i
 \Phi(z)\}=\ker \Phi(z),
\]
and, similarly, $\ker \Phi(z)=\ker \cF(z)$.

\item It is not difficult to see that the condition
$0\in \rho({\mathbf N}_{\Phi,\Psi}(z,z))$ implies that $\ker
\Phi(z)=\ker \Psi(z)= 0$ and that $\ran \Phi(z)$ and $\ran \Psi(z)$
are closed. Moreover,
\[
(\ran \Phi(z))^\perp=(\dom \cF(z))^\perp=\mul \cF(z)^*=\mul
\cF(z)=\ker \Phi(z)=\{0\}.
\]
Consequently, $\ran \Phi(z_0)=\cH$ and hence $0\in \rho(\Phi(z))$.
Then in view of (i) we have also $0\in \rho(\Psi(z))$. The
Nevanlinna kernel for $\cF(\cdot)$ as defined in \eqref{eq:Kern_F}
is given by
\[
    {\mathbf N}_\cF(z,w) 
    =(\Phi(w))^{-*}
    ({\mathbf N}_{\Phi,\Psi}(z,w))(\Phi(z))^{-1},    \quad z\ne\bar w\,\,
    z,w\in\cmr,
\]
and there is a similar formula for $-\cF(\cdot)^{-1}$. 
Therefore,
\[
 0\in \rho({\mathbf N}_{\Phi,\Psi}(z,z)) \quad \Leftrightarrow\quad
 \cF(\cdot)\in R^u[\cH] \quad \Leftrightarrow\quad -\cF(\cdot)^{-1}\in R^u[\cH].
\]
\end{enumerate}
\end{remark}



\section{Invariance theorems for harmonic operator-valued functions and quadratic forms}

Let us recall from \cite[Section 7.1]{Kato} the definition of a
boundedly holomorphic function $T(\cdot)$ with values in the set
$\cC(\cH)$ of closed (not necessarily bounded) operators acting in
$\cH$.
  \begin{definition}\label{def_Kato}
Let $T(\kappa)$ be a family of operators with values in $\cC(\cH)$
and defined in a neighborhood of $\kappa_0\in\mathbb C$ and let
$\zeta\in\rho\bigl(T(\kappa_0)\bigr)$. The family $T(\cdot)$ is
called holomorphic at $\kappa_0\in\C$ if
$\zeta\in\rho\bigl(T(\kappa)\bigr)$ and the resolvent
$R(\zeta,\kappa)=\bigl(T(\kappa)-\zeta\bigr)^{-1}$ is boundedly
holomorphic in $\kappa$ for $|\kappa-\kappa_0|$ small enough.
    \end{definition}

It is shown in \cite[Theorem 7.1.3]{Kato} that in this case the resolvent
$R(\zeta,\kappa)$ of the family $T(\kappa)$  is holomorphic in both variables
$(\zeta,\kappa)$ in an appropriate domain in $\C^2$.

The following definition of holomorphic $R$-function is crucial in the sequel.
   \begin{definition}\label{def_strong_holomorphy}
A function $F\in R(\cH)$ will be called a \emph{strongly
holomorphic function} 
if the following two conditions are satisfied:
  \begin{enumerate}
  \item [(i)] the set
\begin{equation}\label{6.2A}
 \mathcal D(F) := \cap_{z\in\mathbb C_+}\dom F(z)\qquad \text{is dense in}\qquad  \cH.
\end{equation}
  \item [(ii)] vector function $F(z)u$ is holomorphic in a domain
$\Omega\subset\mathbb C_+\cup \R$ for each $u\in\mathcal D(F)$.
\end{enumerate}
\end{definition}

  \begin{remark}
Note that in general, the domain of holomorphy $\Omega$ in $(ii)$
might be broader than the corresponding domain in Definition
\ref{def_Kato}. Namely,  in general conditions (i) and (ii) do not
imply the local boundedness of the resolvent $(F(\cdot)+i)^{-1}$ at
real points. For instance, consider the function
   \begin{equation*}
F(z)=Az, \quad 0 \le A=A^*\in \cC(\cH)\setminus {\mathcal B}(\cH)
,\quad  \dom F(z)=\dom A \not = \cH, \quad z\in\mathbb
C\setminus\{0\},
\end{equation*}
and $\dom (F(0))=\cH$. It is easily seen that $\IM F(z)= Ay\ge 0$
for $z=x+iy\in\mathbb C_+$ and conditions  (i), (ii) are satisfied
with $\Omega = \mathbb C_+\cup \R$  and $\mathcal D(F)= \dom A$.

However, $F(\cdot)$ is not holomorphic at zero in the sense of Definition
\ref{def_Kato}.  Indeed, $F(z)^{-1} = z^{-1}A^{-1}$ is not boundedly holomorphic at zero
even if $A^{-1}\in\cB(\cH)$.

This example  shows that the domain of holomorphy $\Omega$ in Definition
\ref{def_strong_holomorphy}$(ii)$ might be broader than the corresponding domain in
Definition \ref{def_Kato}.
    \end{remark}

An unbounded Nevanlinna function is not in general strongly
holomorphic. In fact, in the next example an extreme situation of a
Nevanlinna function is constructed, such that the domains of
$F(\lambda)$ and $F(\mu)$ have a zero intersection:
\[
 \dom F(\lambda)\cap \dom F(\mu)=\{0\} \text{ for all }
 \lambda,\mu \in \cmr.
\]

\begin{example}\label{Ex4A}
Let $A$ and $B\ge 0$ be two bounded selfadjoint operators on a
Hilbert space $\cH$ with $\ker A=\ker B=\{0\}$ and such that
\[
 \ran A\cap \ran B=\{0\}.
\]
Then $M(z)=A-\frac{1}{z}\, B$, $z\neq 0$, is a Nevanlinna function
from the class $R[\cH]$. The transform $F(z):=-M(z)^{-1}$ is an
operator-valued Nevanlinna function in the class $R(\cH)$:
\[
 F(z):=-\left(A-\frac{1}{z}\, B\right)^{-1}, \quad z\in\cmr.
\]
To consider the domain of $F(z)$ at two points $\lambda,\mu\in\cmr$
and assume that there is a nonzero vector $k\in \dom
F(\lambda)\cap\dom F(\mu)$. This means that $k=M(\lambda)f=M(\mu)g$
for some $f,g\in\cH$, i.e.,
\[
 (A -\frac{1}{\lambda}\,B)f=(A -\frac{1}{\mu}\,B)g
\]
or, equivalently,
\[
 A(f -g)=B \left(\frac{1}{\lambda}f-\frac{1}{\mu}g\right).
\]
Since $\ran A\cap\ran B=\{0\}$ and $\ker A=\ker B=\{0\}$, we get
$f=g$ and $\frac{1}{\lambda}f=\frac{1}{\mu}g$ which leads to
$\lambda=\mu$. Therefore,
\[
  \dom F(\lambda)\cap \dom F(\mu)=\{0\}, \quad \lambda\neq \mu,\,\,
 \lambda,\mu \neq 0.
\]

Recall that there exist bounded nonnegative selfadjoint operators on
a Hilbert space $\cH$ with $\ker A=\ker B=\{0\}$ and such that
\[
 \ran A\cap \ran B=\{0\}, \quad \ran A^{1/2}=\ran B^{1/2};
\]
see e.g. \cite[Example, p.278]{FW} for an example of such operators.
Then it follows from $\ran A^{1/2}=\ran B^{1/2}$ that there exists a
bounded and boundedly invertible positive operator $C$ such that
\[
 A=B^{1/2} C B^{1/2}.
\]
Hence, this choice of $A$ and $B$ implies that $M(z)=B^{1/2} (C-
\frac{1}{z})B^{1/2}$ and
\[
 F(z)=-B^{-1/2} \left(C- \frac{1}{z}\right)^{-1}B^{-1/2}=B^{-1/2}\wt F(z)B^{-1/2},
\]
where $\wt F(z)=-(C- \frac{1}{z})^{-1}$ satisfies, $\wt
F(\cdot),-\wt F^{-1}(\cdot)\in R^u[\cH]$; cf. Remark
\ref{later0}~(iii). Consequently, for every $z\in \cmr$ the form
\[
 \st_{F(z)}[u,v]:=\frac{1}{z-\bar z}\,[(F(z)u,v)-(u,F(z)v)]
 =({\mathbf N}_{\wt F}(z,z)B^{-1/2}u,B^{-1/2}v), \quad u,v\in \dom F(z),
\]
is closable and its closure has the same formula which is defined on
a constant domain $\ran B^{1/2}$. Nevanlinna functions with this
property are studied systematically in a forthcoming paper
\cite{DHM15} by the authors and they are called form-domain
invariant Nevanlinna functions.
\end{example}

The general definition of the class of form-domain invariant
Nevanlinna families $\cF(\cdot)\in \wt R(\cH)$ reads as follows; cf.
\cite{DHM15} for a treatment of such functions as Weyl functions of
boundary triplets and boundary pairs (or boundary relations).

\begin{definition}\label{Nevforminv}
A Nevanlinna family $\cF(\cdot)\in \wt R(\cH)$ is said to be
form-domain invariant if its operator part $F_s(\cdot)\in R(\cH_s)$
is form-domain invariant, which means that the quadratic form
$\st_{F_s(\lambda)}$ in $\cH_s$ generated by the imaginary part of
$F_s(\lambda)$ via
\[
 \st_{F_s(\lambda)}[u,v]=\frac{1}{\lambda-\bar\lambda}\,[(F_s(\lambda)u,v)-(u,F_s(\lambda)v)],
\]
is closable for all $\lambda\in\cmr$ and the closure of the form
$\st_{F_s(\lambda)}$ has a constant domain.
\end{definition}

In what follows the set of nonnegative harmonic functions in
$\mathbb C_+$ is denoted by $Har_+(\mathbb C_+)$. In this section we
will systematically make use of the classical Harnack's inequality:
Given a pair of points $z_1, z_2\in\mathbb C_+$, 
there exists positive constants $c_j = c_j(z_1,z_2),\  j\in\{1,2\},$ such that
     \begin{equation}\label{6.0}
c_1 h(z_1)\le h(z_2)\le c_2 h(z_1), \qquad h(\cdot) \in
Har_+(\mathbb C_+).
     \end{equation}
It is emphasized that the constants $c_1$ and $c_2$ do not depend on
$h(\cdot)\in Har_+(\mathbb C_+)$.

With any strongly holomorphic $R$-function $F(\cdot)\in R(\cH)$ one
associates a family of the quadratic forms $\mathfrak t(z)[\cdot]$
given by
      \begin{equation}\label{6.1}
\mathfrak t(z)[u] := \IM (F(z)[u]) := \IM\bigl(F(z)u, u\bigr)\ge 0,
\qquad u\in\dom\bigl(F(z)\bigr), \quad z\in\mathbb C_+.
      \end{equation}
Equipping $\dom F(z)$ with the inner product
\[
(u, v)_{+,z} = (u, v)_{\cH} + \mathfrak t(z)[u, v],\qquad u,
v\in\dom F(z),
\]
we obtain a pre-Hilbert space $\cH'_+(z)$. The corresponding energy
space is denoted by $\cH_+(z)$, i.e., it is the completion of
$\cH'_+(z)$ with respect to the norm $\|\cdot\|_{+,z}$.

Recall that the form $\mathfrak t(z)$ is called closable if the
norms $\|\cdot\|_{\cH}$ and $\|\cdot\|_{+,z}$ are compatible. The
latter means that the completion  of $\cH'_+(z)$ holds within $\cH,$
i.e. $\cH_+(z)\subset\cH.$  The form $\mathfrak t(z)$ is called
closed if it is closable and  $\cH_+(z) = \cH'_+(z)$, i.e.
$\cH'_+(z)$ is a Hilbert space.

In the following proposition we investigate  certain  stability
properties of the family \eqref{6.1}.
   \begin{proposition}\label{prop_form_stability}
Let $F(\cdot)\in R(\cH)$ be strongly holomorphic. Assume in addition
that for some $z_0\in\mathbb C_+$ the form $\mathfrak t(z_0)$ is
closable and  $\mathcal D(F)$ is a core for its closure $\overline
{\mathfrak t}(z_0)$. Then:

\item[\;\;\rm (i)]
 The form $\mathfrak t(z)$ is closable for any $z\in\mathbb C_+;$

\item[\;\;\rm (ii)]
$\mathcal D(F)$ is a core for  the closure $\overline {\mathfrak
t}(z)$ of the form ${\mathfrak t}(z)$ for each $z\in\mathbb C_+.$
Moreover,  the corresponding energy spaces $\cH_+(z)$, $z\in\mathbb
C_+,$ coincide algebraically and topologically,
  \begin{equation}\label{4.6}
\cH_+(z) = \cH_+(z_0), \qquad z\in\mathbb C_+.
  \end{equation}
In particular,
  \begin{equation}\label{4.6A}
\mathcal D[F] := \cap_{z\in \C_+} \cH_+(z) = \cH_+(z_0). 
  \end{equation}
\item[\;\;\rm (iii)]  For any pair $u,v\in \mathcal D[F]$ the function  $\overline {\mathfrak t}(\cdot)[u,v]$ is a
harmonic (hence real analytic) function in $\C_+$.
   \end{proposition}
\begin{proof}
(i) and (ii).  Lets us show that the form $\mathfrak t_z$ is
closable for any $z\in\mathbb C_+$. Assume that $h_n\in\mathcal D$
and
    \begin{equation}\label{6.3}
\lim_{n\to\infty}\|u_n\| = 0 \quad  \text{and} \quad \mathfrak
t(z)[u_n - u_m]\to 0\quad  \text{as} \quad n,m\to\infty.
  \end{equation}
Then by the Harnack's inequality there exist positive constants $c_1 = c_1(z_0,z),$
$c_2= c_2(z_0,z),$ depending only on $z_0,z,$ and such that
    \begin{equation}\label{6.4}
0\le  c_1\mathfrak t(z_0)[u_n - u_m] \le \mathfrak t(z)[u_n - u_m]
\le c_2 \mathfrak t(z_0)[u_n - u_m], \quad u_n\in\mathcal D(F).
    \end{equation}
Combining \eqref{6.3} with the first inequality in  \eqref{6.4} we
get $ \mathfrak t(z_0)[u_n - u_m]\to 0$ as $n,m\to\infty$. Since in
addition  $\lim_{n\to\infty}\|u_n\| = 0$ and the form $\mathfrak
t(z_0)$ is closable, one has $\mathfrak t(z_0)[u_n]\to 0$ as
$n\to\infty$. In turn, the second  inequality in \eqref{6.4}  with
$u_m=0$ yields
 \begin{equation}\label{6.5}
\lim_{n\to\infty}\mathfrak t(z)[u_n] = 0,\qquad  z\in\mathbb C_+.
 \end{equation}
Thus, \eqref{6.3} implies \eqref{6.5}.   This means the closability
of $\mathfrak t_z$, hence the identical embedding
$\cH'_+(z)\hookrightarrow\cH$ is extended to a (continuous)
embedding of the energy space $\cH_+(z)$ into $\cH$. Moreover, it
follows from \eqref{6.4} that the norms in $\cH'_+(z)$ and
$\cH'_+(z_0)$ are equivalent. Hence completing the spaces
$\cH'_+(z)$ and $\cH'_+(z_0)$ we conclude that $\cH_+(z)$ and
$\cH_+(z_0)$ coincide algebraically and topologically.

(iii) Since $\cD(F)$ is a core for  $\overline {\mathfrak t}(z_0)$,
it follows form the definition of the closure that for any $u\in
\mathcal D[F] = \cH_+(z_0)$ there exists $u_n\in \cD(F)$ such that
$$
\overline {\mathfrak t}(z_0)[u] = \lim_{n\to\infty} \mathfrak
t(z_0)[u_n], \quad u_n\in\mathcal D(F).
$$
It follows from \eqref{6.4} (with $u_m=0$) that  $\overline
{\mathfrak t}(z)[u] = \lim_{n\to\infty} \mathfrak t(z)[u_n]$
uniformly on compact subsets of $\C_+.$ So, by the first Harnack
theorem, $\overline {\mathfrak t}(\cdot)[u]$ is a nonnegative
harmonic function in $\C_+.$ Using the polarization identity one
proves that $\overline {\mathfrak t}(\cdot)[u,v]$ is also harmonic
for any pair $u, v\in \mathcal D[F].$
  \end{proof}
Proposition \ref{prop_form_stability}  makes it possible to introduce the imaginary part
of the function $F(\cdot)$.
  \begin{definition}\label{def_imaginary_part}
Let $F(\cdot)\in R(\cH)$ satisfy the conditions of Proposition
\ref{prop_form_stability}.   Denote by $F_I(z)$, $z\in \mathbb C_+,$
the  nonnegative self-adjoint  operator associated with the closed
form $\overline {\mathfrak t}(z)$ in accordance with the first
representation theorem (see \cite[Theorem 6.2.1]{Kato}).
   \end{definition}
Note that the operator $F_I(z) = F_I(z)^*$  is a self-adjoint extension of the operator
$$
F'_I(z) :=  (F(z) -F(z)^*)/2i, \quad  \dom(F'_I(z)) =  \dom(F(z))\cap \dom(F(z)^*),
$$
which is  only nonnegative symmetric  not necessarily essentially
self-adjoint.
   \begin{remark}
(i) Note that in accordance with the second representation theorem (see \cite[Theorem
6.2.23]{Kato})) and Definition \ref{def_imaginary_part}  equalities
\eqref{4.6}-\eqref{4.6A} can be rewritten as
  \begin{equation}\label{4.11}
\cH_+(z) = \dom(F_I(z)^{1/2}) = \dom(F_I(z_0)^{1/2}) = \cH_+(z_0) = \cD[F]\quad
\text{for each}\quad z\in \C_+.
  \end{equation}
Here the spaces $\cH_+(z)$ and $\dom(F_I(z)^{1/2})$ (equipped with the graph norm)
coincide algebraically and topologically.

(ii) Proposition  \ref{prop_form_stability}  shows  that  the family $F(\cdot)$ is   a
holomorphic family in $\C_+$ of the type $(B)$ in the sense of T. Kato \cite[Section
7.4.2]{Kato}
     \end{remark}
    \begin{proposition}\label{prop6.2}
Let $F(\cdot)\in R(\cH)$ and let the conditions of Proposition \ref{prop_form_stability}
be satisfied.
Assume also that  $F_I(z_0)\in \cB(\cH)$ for $z_0\in\mathbb C_+$.
Then:

\item[\;\;\rm (i)]  $F_I(\cdot)$ takes values in $\cB(\cH);$

\item[\;\;\rm (ii)]   The function  $F(\cdot)$ admits a representation
     \begin{equation}\label{6.9A}
F(z) =  G(z) + T \quad z\in \C_+, 
     \end{equation}
where $G(\cdot)\in R[\cH]$ and   $T=T^*\in \cC(\cH)$.
       \end{proposition}
   \begin{proof}
(i)  Since $F\in R(\cH)$ is a strongly holomorphic function, the family
$$
\mathcal F=\{\IM\bigl(F(\cdot)u,u\bigr): u\in\mathcal D(F)\}
$$
is well defined and constitutes the family of  nonnegative harmonic
functions, $\mathcal F\subset Har_+(\mathbb C_+)$. Fix $z\in\mathbb
C_+$. Then, by Proposition \ref{prop_form_stability}(ii)  the form $
t(z)[\cdot]$ is closable  and by  the Harnack's inequality
\eqref{6.0},
     \begin{equation*}
     0\le \mathfrak t(z)[u] := \IM\bigl(F(z)u, u\bigr) \le  c_2\IM\bigl(F(z_0)u, u\bigr)
\le c_2\|F_I(z_0)\|\cdot\|u\|^2, \quad u\in \mathcal D(F).
     \end{equation*}
It follows that the form $\mathfrak t(z)$ is bounded on $\mathcal
D(F)$. Since $\mathcal D(F)$ is a core for $\mathfrak t(\cdot)$, the
form $\mathfrak t(z)$ admits a bounded continuation on $\cH$ and by
the Riesz representation theorem,
     \begin{equation}\label{6.11A}
\mathfrak t(z)[u, v] = \bigl(T(z)u, v\bigr)_{\cH}, \qquad 0\le T(z)
= T^*(z) \in \cB(\cH), \quad u,v\in \cB(\cH).
  \end{equation}
Using the polarization identity we obtain from \eqref{6.1}  that
\[
\mathfrak t(z)[u, v]= (2i)^{-1}\left(\bigl(F(z)u, v\bigr) - \bigl(u,
F(z)v\bigr)\right), \qquad u, v \in \dom F(z).
\]
Combining this identity with \eqref{6.11A}  we derive
\[
\bigl((F(z)-i T(z))u, v\bigr) = \bigl(u, (F(z) - iT(z))v\bigr),\qquad  u,v \in \dom F(z).
\]
Since $\dom F(z)$ is dense in $\cH$, it follows that $v\in
\dom\bigl(F(z)^* + iT(z)\bigr)$, i.e.
\[
\dom F(z)\subset
\dom\bigl(F(z)^* + iT(z)\bigr).
\]
On the other hand,
since $T(z) = T^*(z)$ is
bounded, then
\[
\dom\bigl(F(z)^* + iT(z)\bigr)= \dom F(z)^*\quad\mbox{ and }\quad
\dom
F(z)\subset \dom F(z)^*.
\]
By symmetry,
\[
\dom F(z)^* = \dom F(\overline z) \subset \dom F^*(\overline z) = \dom F(z)
\]
Thus, $\dom F(z)^* = \dom F(z)$ and the imaginary part $F_I(\cdot) := (2i)^{-1}\bigl(F(\cdot)-F^*(\cdot)\bigr)$ \\
of $F(\cdot)$  is well defined and
   \begin{equation}\label{6.12}
F_I(z)u = T(z)u, \quad   u\in\dom F(z) (\supset \mathcal D(F)).
   \end{equation}
Hence $F_I(z)$ is bounded and its closure coincides with $T(z)$.

(ii)  Being a nonnegative harmonic $\cB(\cH)$-valued function in
$\mathbb C_+$, $T(\cdot)$ admits a representation
   \begin{equation}\label{Int_rep-n}
T(z)= B_0 + B_1 y + \int_{\mathbb R}\frac{y}{(x-t)^2+y^2}d\Sigma(t),
   \end{equation}
where $B_j = B_j^*\in\cB(\cH)$, $j\in \{0,1\}$, $B_1\ge 0$, and   $\Sigma(\cdot)$ is the $\cB(\cH)$-valued operator measure
satisfying
     \begin{equation}\label{Measuer_Condition}
K_\Sigma := \int_{\mathbb R}(1+t^2)^{-1}d\Sigma(t)\in\cB(\cH).
     \end{equation}\label{6.13}
Define $R[\cH]$-function $G(\cdot)$ by setting
    \begin{equation}\label{6.14}
G(z) = B_0 + B_1 z   + \int_{\mathbb
R}\left(\frac{1}{t-z}-\frac{1}{1+t^2}\right)\,d\Sigma(t).
    \end{equation}
Further we let
 \begin{equation}\label{6.18}
G_1(z) := F(z) - G(z)
\end{equation}
and note that $G_1(\cdot)$ is holomorphic in $\mathbb C_+$.
Moreover, it follows from \eqref{6.12}, \eqref{6.13} and
\eqref{6.14} that
\[
\IM \bigl(G_1(z)u, u\bigr) = 0, \quad     z\in\mathbb C_+, \quad
u\in\mathcal D(F).
\]
Hence  the operator $G_1(z)$ is symmetric for any $z\in\mathbb C_+$.
Let us show that $G_1(z)$ is self-adjoint. Since $F(z)$ is
$m$-dissipative for  $z\in\mathbb C_+$ and $G(\cdot)$ takes values
in $\cB(\cH),$ one has $\rho\bigl(G_1(z)\bigr)\cap\mathbb C_- \not =
\emptyset$.

Further, since $F(z)^*$ is $m$-accumulative for $z\in\mathbb C_+$
and $G(z)^*\in\cB(\cH)$, we get
\[
G_1(z)^*=F(z)^*-G(z)^* \qquad  \text{and}\qquad
\rho\bigl(G_1(z)^*\bigr)\cap\mathbb C_+\not = \emptyset.
\]
Thus, $G_1(z) = G_1(z)^*$  for any $z\in\mathbb C_+$. Being
holomorphic in $\mathbb C_+$, the operator-valued function
$G_1(\cdot)$ is constant, $G_1(z)=T=T^*,\ z\in\mathbb C_+$.
Combining this with \eqref{6.18} leads to \eqref{6.9A}.
  \end{proof}
Next we investigate the invariance property of real continuous spectrum.

Recall that $\lambda_0\in \sigma_c(T)$ if $\lambda_0\not \in \sigma_p(T)$ and  there
exists a non-compact (quasi-eigen) sequence $f_n\in \dom(T)\subset \cH$ such that
$$
\lim_{n\to\infty} \|(T- \lambda_0)f_n\| =0.
$$
  \begin{proposition}\label{prop_Harmonic_cont_and_point_spec}
Let $F\in R(\cH)$  and satisfy  the conditions of Proposition \ref{prop_form_stability}.
Let also $F_I(\cdot)$ be  its imaginary part in the sense of  Definition
\ref{def_imaginary_part}. Then the following holds:

\item[\;\;\rm (i)]
If  $a = \overline{a}\in\sigma_c\bigl(F_I(z_0)\bigr)$ for some
$z_0\in\mathbb C_+$,  then
\[
a\in\sigma_c\bigl(F_I(z)\bigr)\qquad   \text{for}\qquad  z\in\mathbb
C_+;
\]

\item[\;\;\rm (ii)]
If  $a = \overline{a}\in\sigma_p\bigl(F_I(z_0)\bigr)$ for some
$z_0\in\mathbb C_+$, then
\[
a\in\sigma_p\bigl(F_I(z)\bigr)\quad\text{and}\quad
ker\bigl(F_I(z)-a)= \ker\bigl(F_I(z_0)-a)\quad  \text{for}\quad
z\in\mathbb C_+.
\]
       \end{proposition}
   \begin{proof}
(i)  Without loss of generality we can assume that $a=0.$ Since
$0\in\sigma_c\bigl(F_I(z_0)\bigr)$, there exists an non-compact
(quasi-eigen) sequence $\{v_k\}_{k\in \mathbb N}\in \dom(F_I(z_0))$
such that
       \begin{equation}\label{4.25}
\|v_k\|=1\quad \mbox{ and }\quad \lim_{k\to\infty}\|F_I(z_0)v_k\|=0.
     \end{equation}
By Proposition \eqref{prop_form_stability}(ii),  $\{v_k\}_{k\in \N}\in
\dom(F_I(z_0))\subset \cD[F]$. Using Definition \ref{def_imaginary_part} and relation
\eqref{4.11}  one rewrites the right-hand side of inequality \eqref{6.4} as
    \begin{equation}\label{6.4New}
0  \le  \|F_I(z)^{1/2}v_k\|^2 = \mathfrak t(z)[v_k] \le c_2
\mathfrak t(z_0)[v_k]= \|F_I(z_0)^{1/2}v_k\|^2, \quad v_k\in\mathcal
D[F]= \cH_+(z).
    \end{equation}
Combining \eqref{4.25} with  \eqref{6.4New} and noting that the sequence  $\{v_k\}$ is
not compact, one gets that $0\in\sigma_c\bigl(F_I(z)^{1/2}\bigr)$. Hence
$0\in\sigma_c\bigl(F_I(z)\bigr)$.

(ii)  Let  $a=0 \in\sigma_p\bigl(F_I(z_0)\bigr)$ and $u\in \ker\bigl(F_I(z_0)\bigr)$.
Hence $u\in \ker\bigl(F_I(z_0)^{1/2} \bigr)$.  By Proposition~\eqref{prop_form_stability}~(ii),  $u\in \dom(F_I(z))\subset \cD[F] =
\dom(F_I(z_0)^{1/2})$ for each $z\in \C_+$. It follows from \eqref{6.4New} with $u$ in
place of $v_k$ that $F_I(z)^{1/2}u = 0$. Hence $F_I(z)u = 0$ for each $z\in \C_+$.
 \end{proof}

Next we slightly improved Proposition \ref{prop_Harmonic_cont_and_point_spec}(i).
  \begin{proposition}\label{prop_Harmonic_cont_spec}
Let $F\in R(\cH)$  and satisfy  the conditions of Proposition
\ref{prop_form_stability}. Let also $F_I(\cdot)$ be  its imaginary
part in the sense of  Definition \ref{def_imaginary_part}. If  $a =
\overline{a}\in\sigma_c\bigl(F_I(z_0)\bigr)$ for some $z_0\in\mathbb
C_+$,  then
  \begin{equation*}
a\in\sigma_c\bigl(F_I(z)\bigr)\qquad   \text{for}\qquad  z\in\mathbb
C_+.
  \end{equation*}
Moreover, a quasi-eigen sequence can be chosen to be common for all
$F_I(z),\ z\in \mathbb C_+.$
        \end{proposition}
   \begin{proof}
(i) First we reduce the proof to the case of $R$-function with bounded imaginary part.

In fact, this sequence  $\{v_k\}_{k\in \mathbb N}\in \dom(F_I(z_0))$
in \eqref{4.25} can be chosen to be orthonormal. Passing if
necessary to a subsequence $\{v_{n_k}\}_{k\in\mathbb N},$ we can
assume that
\[
\sum^{\infty}_{k=1}\|F_I(z_0)v_k\|^2 := \alpha_{F}(z_0)<\infty.
\]
Since $\mathcal D(F)$ is a core for the form $\mathfrak t(z_0)$, it
is dense  in $\dom(F_I(z_0))(\subset \cH)$ equipped with the graph's
norm. So, there exists a sequence $\{u_k\}^{\infty}_1\subset\mathcal
D(F)$ such that
   \begin{equation}\label{4.26}
\sum^{\infty}_{k=1}\|u_k - v_k\|^2<\infty \quad \text{and}\quad
\sum^{\infty}_{k=1}\|F_I(z_0)(u_k - v_k)\|^2 =: \beta_F(z_0)<\infty.
   \end{equation}
Let  $\cH_0$  be  a subspace spanned by the sequence $\{u_k\}_1^\infty$,   $\cH_0:=
\span\{u_k\}^{\infty}_1$. It is known (see \cite[Theorem 6.2.3]{GK65}) that the system
$\{u_k\}^{\infty}_1$  forms (after possible replacement  of a finite number of vectors
by another system of linearly independent vectors) a Riesz basis
in $\cH_0$. 
Assume for convenience  that such replacement is not needed, i.e. the system
$\{u_k\}_1^\infty$ itself  constitutes  the Riesz basis  in $\cH_0$. Denote by $P_0$ the
orthoprojection in $\cH$ onto  $\cH_0$  and put
     \begin{equation*}
F_{0}(\cdot) := P_0 F(\cdot)\lceil\cH_0 \quad \text{and}\quad
F_{I,0}(z) := P_0 F_{0,I}(z)\lceil\cH_0, \quad z\in\mathbb C_+.
  \end{equation*}
First we show that $F_{I,0}(\cdot)\in R[\cH_0]$. Indeed, since the
system $\{u_k\}_1^\infty$ forms the Riesz basis  in $\cH_0$, any
$u\in \cH_0$ admits a decomposition  $u=\sum_k c_k u_k$  with  $c:=
\{c_k\}_1^\infty \in l^2(\mathbb N)$. Clearly,
   \begin{eqnarray*}
\|F_{I,0}(z_0)\sum^n_{k=1}c_k u_k \|^2 &\le & \left(\sum^n_{k=1}|c_k|\cdot\|F_{I,0}(z_0)u_k \|\right)^2 \nonumber   \\
&\le &(\sum^n_1|c_k|^2)\bigl(\sum^n_1\|P_0F_I(z_0)u_k\|^2\bigr)\nonumber   \\
&\le & 2(\alpha_{F}(z_0) + \beta_{F}(z_0)))\cdot\|c\|_{l^2}^2, \quad
n\in \mathbb N.
  \end{eqnarray*}
Hence $F_{I,0}(z_0) \in [\cH_0]$. Moreover, it is easily seen that
$F_{0}(\cdot)$ satisfies the conditions of
Proposition~\ref{prop_form_stability} together with $F(\cdot)$.
Thus, by Proposition  \ref{prop6.2}, $F_{I,0}(\cdot)$ takes values
in $\cB(\cH_0)$.

(ii)   It follows from \eqref{4.25} and \eqref{4.26} that the sequence $\{u_k\}_{k\in
\N}$ is a quasi-eigen  sequence for the operator $F_{I,0}(z_0)$ corresponding to the
point $a=0,$ i.e. it is bounded, non-compact and
   \begin{equation}\label{4.27}
\lim_{k\to\infty}\|F_{I,0}(z_0)u_k\|=0.
  \end{equation}
Define  a family of scalar nonnegative harmonic  functions
$h_k(\cdot) := \bigl(F_{I,0}(\cdot)u_k,u_k\bigr)$ in $\mathbb C_+$.
Since the sequence $\{u_k\}_{k\in \N}$ is bounded, relation
\eqref{4.27} yields
\[
\lim_{k\to\infty}h_k(z_0) = \lim_{k\to\infty} \bigl(F_{I,0}(z_0)u_k, u_k\bigr) = 0.
\]
By the Harnack inequality \eqref{6.0}, this relation implies similar
relation  for any $z\in\mathbb C_+$ (cf.~\eqref{6.4}),
\[
\lim_{k\to\infty}h_k(z) = \lim_{k\to\infty} \bigl(F_{I,0}(z)u_k,
u_k\bigr) = \lim_{k\to\infty} \|F_{I,0}^{1/2}(z)u_k\|^2 = 0, \qquad
z\in\mathbb C_+.
\]
Since $F_{I,0}(\cdot)$ takes values in $\cB(\cH_0)$, the latter implies
$\lim_{k\to\infty} \|F_{I,0}(z)u_k\| =0$ which proves the result.
    \end{proof}
%
%
  \begin{corollary}\label{cor6.4}
Let $F(\cdot)\in R(\cH)$ and $F(z_0)\in \cB(\cH)$ for some
$z_0\in\mathbb C_+$. Then $F(\cdot)\in R[\cH]$.
  \end{corollary}
  \begin{proof}
By Proposition \ref{prop6.2},  $F(\cdot)= F_1(\cdot) + T$ where $F_1(\cdot)\in R[\cH]$ and $T=T^*$. Since
\[
F(z_0)=F_1(z_0) + T\in\cB(\cH)\quad\mbox{ and }\quad
F_1(z_0)\in\cB(\cH),
\]
the operator
$T$ is bounded and $F\in R[\cH]$.
 \end{proof}

Next we present another proof of statement (iv) in Theorem~\ref{invar}.
  \begin{proposition}\label{prop6.5}
Let $F(\cdot)\in R(\cH)$, $a
=\overline{a}\in\sigma_p\bigl(F(z_0)\bigr)$ for $z_0\in\mathbb C_+$.
Then
\[
a\in \sigma_p \bigl(F(z)\bigr)\qquad \text{and}\qquad
\ker\bigl(F(z)-a)= \ker\bigl(F(z_0)-a) \qquad  z\in\mathbb C_+.
\]
\end{proposition}
   \begin{proof}
Since $F(\cdot) - a\in R(\cH)$ for any $a\in \mathbb R$, we can
assume without loss of generality that $a=0$. Let us put $G(\cdot) :
= -\bigl(F(\cdot)+ i\bigr)^{-1}$. Since $F(z)$ is $m$-dissipative
for $z\in \mathbb C_+$, $G(\cdot) \in R[\cH]$. Moreover, due to the
classical estimate
\[
\|G(z)\| = \|\bigl(F(z) + i\bigr)^{-1}\| \le 1, \qquad z\in \C_+,  
\]
$G(\cdot)$ is a contractive holomorphic operator-valued function in
$\mathbb C_+.$
%
%
%
 Hence its  imaginary part $G_I(\cdot)$  is also contractive, $0\le G_I(z) \le 1,\  z\in\mathbb C_+$.
Further, let us assume that  $u_0 \in \ker\bigl(F(z_0)\bigr)$ and
for definiteness $\|u_0\|=1$.  Then   $h(\cdot) :=
\bigl(G_I(\cdot)u_0,u_0\bigr)$ is a scalar nonnegative contractive
harmonic function in $\mathbb C_+$. Moreover,  since
$\bigl(F(z_0)+i\bigr)u_0 = iu_0$   we get
    \begin{equation*}
G_I(z_0)u_0=u_0\quad \text{and}\quad   h(z_0) = \bigl(G_I(z_0)u_0,u_0\bigr) = \|u_0\|^2 = 1.
    \end{equation*}
According to the Maximum  Principle applied to  the contractive
harmonic function $h(\cdot)$, one gets  $h(z) = h(z_0) = 1$,
$z\in\mathbb C_+$. Rewriting this identity in the form
      \begin{equation*}
\bigl((I - G_I(z))u_0, u_0\bigr) = 0,\quad  z\in\mathbb C_+,
      \end{equation*}
and noting that  $I-G_I(z)\ge 0$, we derive  $G_I(z)u_0=u_0,\
z\in\mathbb C_+$. Since $G(\cdot)$ is contractive, the previous
identity yields   $G(z)u_0 = iu_0$, i.e. $F(z)u_0=0$ for
$z\in\mathbb C_+$.
\end{proof}
   \begin{corollary}
Assume the conditions of Proposition \ref{prop6.5} and let $\cH_0
:=\ker\bigl(F(i)-a\bigr)$. Then $\cH = \cH_0\oplus\cH_1$  and
$F(\cdot)$ admits the following orthogonal decomposition
    \begin{equation*}
F(z)=aI_{\cH_0}\oplus F_a(z),
      \end{equation*}
where $F_a(\cdot)\in R(\cH_1)$ and $\ker\bigl(F_a(z)-a\bigr)=\{0\},\
z\in\mathbb C_+$.
 \end{corollary}
  \begin{proposition}\label{prop_continuous_spec}
Let $F\in R(\cH)$ satisfy the conditions of Proposition
\ref{prop_form_stability}, and let $a =
\overline{a}\in\sigma_c\bigl(F(z_0)\bigr)$ for some $z_0\in\mathbb
C_+$. Then
  \begin{equation*}
a\in\sigma_c\bigl(F(z)\bigr)\qquad   \text{for}\qquad  z\in\mathbb
C_+.
  \end{equation*}
Moreover, a quasi-eigen sequence can be chosen to be common for all
$F(z),\ z\in \mathbb C_+.$
        \end{proposition}
   \begin{proof}
Repeating the  procedure  applied in the proof of Proposition
\ref{prop_Harmonic_cont_spec}  one reduces the proof to the case of $F=F_0$ with values
in $\cB(\cH).$
 Therefore $F_0(\cdot)$ admits the  integral representation \eqref{6.14}
%
%
%
%
with  $B_0\ge 0,\  B_1=B^*_1\in[\cH_0]$ and $\Sigma(\cdot)$ being  the
$\cB(\cH_0)$-valued operator measure  satisfying condition \eqref{Measuer_Condition}.
%
%
%
%
Clearly, $K_\Sigma \ge 0$ and  $K_\Sigma\in \cB(\cH_0)$.

Repeating the reasoning of Proposition  \ref{prop_Harmonic_cont_spec}  one shows that
there exists a quasi-eigen sequence for $F_{0}(z_0)$ and such that $\{u_k\}\subset
\cD(F)$, i.e.
   \begin{equation}\label{4.27NevFun}
\lim_{k\to\infty}\|F_{0}(z_0)u_k\|=0.
  \end{equation}

Setting
\[
H(z):= \IM F_0(z)  := (2i)^{-1}\bigl(F_0(z) -
F_0(z)^*\bigr)
\]
one  defines a family of scalar nonnegative harmonic  functions
$h_n(\cdot) := \bigl(H(\cdot)u_n,u_n\bigr)$ in $\mathbb C_+$. It
follows from \eqref{4.27NevFun} that
%
$\lim_{n\to\infty}h_n(z_0) = \lim_{n\to\infty} \bigl(H(z_0)u_n, u_n\bigr) = 0.$
%
%
By Proposition \ref{prop_Harmonic_cont_spec}  similar conclusion
holds   for any $z\in\mathbb C_+$, i.e.
     \begin{equation}\label{6.30}
\lim_{n\to\infty}h_n(z) = \lim_{n\to\infty} \bigl(H(z)u_n, u_n\bigr)
= 0, \qquad z\in\mathbb C_+.
     \end{equation}
On the other hand, it follows from \eqref{Int_rep-n}  with account of
\eqref{Measuer_Condition}  that
    \begin{equation}\label{6.32}
H(i) = -i\bigl(F_0(i) - B_0\bigr) = B_1 + K_{\Sigma}.
    \end{equation}
Combining \eqref{6.30} with \eqref{6.32} and noting that $B_1\ge 0$ and $K_{\Sigma}>0$,
one gets
     \begin{equation}\label{6.33}
\lim_{n\to\infty}\|B_1u_n\| = \lim_{n\to\infty}\|K_{\Sigma}u_n\| = 0.
 \end{equation}
Further, for any fixed $z\in\mathbb C_+$ we set
    \begin{equation}\label{6.34}
c_2(z):=\max_{t\in\mathbb R}\left|\frac{1+z t}{t-z}\right|.
  \end{equation}
Combining  \eqref{Int_rep-n} with  \eqref{Measuer_Condition}, applying the
Cauchy-Bunyakovskii inequality for integrals,  and taking the notation \eqref{6.34} into
account we derive  (cf. \cite[Section 7]{MalMal03})
     \begin{eqnarray*}
|\bigl((F_0(z)- B_1z - B_0)u, v\bigr)|^2
&\le & \left|\int_{\mathbb R}\frac{1+z t}{(t-z)(1+t^2)}\,d \bigl(\Sigma(t)u, v\bigr)\right|^2 \nonumber   \\
&\le & c_2(z)^2 \int_{\mathbb R} \frac{1}{1+t^2}\,d
\bigl(\Sigma(t)u, u\bigr)
\int_{\mathbb R} \frac{1}{1+t^2}\,d \bigl(\Sigma(t)v, v\bigr) \nonumber   \\
&\le & c_2(z)^2 (K_{\Sigma}u, u)(K_{\Sigma}v, v) \nonumber  \\
&= & c_2(z)^2 \|K_{\Sigma}^{1/2}u\|^2\cdot \|K_{\Sigma}^{1/2}v\|^2.
 \end{eqnarray*}
This "weak" estimate  is equivalent to the following "strong" one
   \begin{equation}\label{6.36}
\|\bigl(F_0(z)- B_1z - B_0\bigr)u\| \le
c_2(z)\|K_{\Sigma}^{1/2}\|\cdot \|K_{\Sigma}^{1/2}u\|,\quad
z\in\mathbb C_+.
 \end{equation}
Inserting in this inequality $u=u_n$ and taking into account \eqref{6.33} yields
     \begin{equation}\label{6.41}
\lim_{n\to\infty} \|\bigl(F_0(z) - B_0\bigr)u_n\| = 0, \qquad
z\in\mathbb C_+.
  \end{equation}
Setting here $z=z_0$ and using the assumption $ \lim_{n\to\infty}\|F_0(z_0)u_n\| = 0,$
we get $\lim_{n\to\infty}B_0u_n =0$. Finally,  combining this relation with \eqref{6.41}
implies
     \begin{equation*}
\lim_{n\to\infty}F_0(z)u_n = 0, \quad   z\in\mathbb C_+.
  \end{equation*}
Since the  sequence $\{u_n\}_{n\in \mathbb N}$ is non-compact, the
latter means that $0\in\sigma_c\bigl(F_0(z)\bigr)$.  Hence
$0\in\sigma_c\bigl(F(z)\bigr)$ and the result is proved.
   \end{proof}
%
%
%

%
%
%
%
%
%
%
%
%
\begin{proposition}
Let $F(\cdot)\in R[\cH]$  and $F(z_0)\in \mathfrak S_p(\cH)$ for
some $z_0\in\mathbb C_+$ and $p\in (0, \infty]$. Then $F(\cdot)$
takes values in $ \mathfrak S_p(\cH)$,
  \begin{equation*}
F(\cdot): \mathbb C_+  \to  \mathfrak S_p(\cH).  
 \end{equation*}
\end{proposition}
\begin{proof}
Since  $F(\cdot)\in R[\cH]$, it  admits integral representation \eqref{6.14}. According
to \eqref{6.36}  the following estimate holds
\[
|\bigl((F(z) - B_1z - B_0)u, v\bigr)|   \le
c_2(z)\|K_{\Sigma}^{1/2}u\|\cdot \|K_{\Sigma}^{1/2}v\|,\quad u,v \in
\cH, \quad   z\in\mathbb C_+,
\]
where $K_{\Sigma}\ge 0$ is  a nonnegative bounded operator in $\cH$
given by \eqref{Measuer_Condition}. This estimate is equivalent to
the following representation
  \begin{equation}\label{6.44}
F(z) - B_1z - B_0 = K^{1/2}_{\Sigma}T(z)K^{1/2}_{\Sigma}, \quad
z\in\mathbb C_+,
  \end{equation}
where $T(z)$ is an operator-valued function with values in
$\cB(\cH)$ and $\|T(z)\|\le c_2(z)$.

On the other hand, setting as above  $H(z):= \IM F(z) := F_I(z)$,
applying the Harnack's inequality \eqref{6.0}, and taking \eqref{6.32}
into account, we obtain
   \begin{eqnarray*}
C_1\bigl(H(z_0)u, u\bigr) &\le & \bigl(H(i)u, u\bigr) = \bigl(B_1u, u\bigr) + \bigl(K_{\Sigma}u, u\bigr) \nonumber   \\
&\le & C_2\bigl(H(z_0)u, u\bigr), \qquad  u\in\cH.
\end{eqnarray*}
It follows that there exists an operator $T_0 \in\cB(\cH)$ with bounded
inverse $T^{-1}_0\in\cB(\cH)$ and such that
  \begin{equation*}
B_1 + K_{\Sigma} = T_0 H(z_0)T_0^* = T_0 F_I(z_0)T_0^*\in\mathfrak
S_p(\cH).
  \end{equation*}
Since both operators $B_1$ and $K_{\Sigma}$ are nonnegative, one
gets
     \begin{equation*}
s_j(B_1) = {\lambda}_j(B_1) \le {\lambda}_j(B_1+K_{\Sigma}) =
s_j(B_1 + K_{\Sigma}), \quad j\in\mathbb N.
  \end{equation*}
Hence $B_1\in\mathfrak S_p(\cH)$ and $K_{\Sigma}\in \mathfrak
S_p(\cH)$. Combining these inclusion with \eqref{6.44} we get
\[
F(z)- B_0\in \mathfrak S_p(\cH).
\]
Setting here $z=z_0$, yields $B_0\in\mathfrak S_p(\cH)$. Thus, $F(z)
\in \mathfrak S_p(\cH)$ for any $z\in\mathbb C_+$.
   \end{proof}
%
%
\begin{corollary}
Let $F(\cdot)\in R[\cH]$  and $F_I(z_0)\in \mathfrak S_p(\cH)$ for
some $z_0\in\mathbb C_+$ and $p\in (0, \infty]$. Then $F_I(\cdot)$
takes values in $ \mathfrak S_p(\cH)$,
  \begin{equation*}
F(\cdot): \mathbb C_+  \to  \mathfrak S_p(\cH).  
 \end{equation*}
\end{corollary}

\begin{remark}\label{rem_general_ideal}
The result is valid for any two-sided ideal $\mathfrak S(\cH)$
instead of $\mathfrak S_p(\cH)$. In particular, for any $F(\cdot)\in
R[\cH]$  the following implication holds
    \begin{equation*}
s_j\bigl(F(z_0)\bigr)= O(j^{-1/p}) \Longrightarrow  s_j\bigl(F(z)\bigr)= O(j^{-1/p}),
\quad z\in \C_+.
    \end{equation*}
\end{remark}

    \section{Examples}
     \begin{example}
Let $\varphi(\cdot)$  be  a scalar $R$-function and $\cH =L^2(0, \infty).$ Consider an
operator-valued function $F_\varphi(\cdot)$ given  by
\[
F_\varphi(z) u = -\frac{d^2u}{dx^2},\qquad  \dom (F_\varphi(z)) =
\{u\in W^2_2(\mathbb R_+) :\ u'(0) = \varphi(z)u(0)\}, \quad z \in
\C_+.
\]
Clearly, $F(\cdot)\in R(\cH)$  and
$$
\mathcal D(F) := \cap_{z\in\mathbb C_+}\dom F(z) = W^2_{2,0}(\mathbb
R_+) := \{u\in W^2_{2}(\mathbb R_+): u(0)= u'(0)=0\}
$$
is dense in $\cH$. The corresponding family of closed  quadratic forms reads as follows
   \begin{equation}\label{quadr_form}
{F}_{\varphi(z)} [u] =  \int_{\R_+}|u'(x)|^2dx + \varphi(z)|u(0)|^2,
\quad u\in \dom ({F}_{\varphi(z)}) = W^1_2(\mathbb R_+),  \quad z
\in \mathbb C_+.
   \end{equation}
However, the imaginary parts of these forms constitute a family of
non-closable (singular) forms
  \begin{equation*}
{\mathfrak t}_{\varphi(z)} [u] := \IM {F}_{\varphi(z)} [u] =
\IM{\varphi(z)}\cdot|u(0)|^2, \quad  u\in W^1_2(\mathbb R_+), \quad
z \in \mathbb C_+.
     \end{equation*}
In accordance with Proposition  \ref{prop_form_stability}  they  are non-closable for
all $z\in \C_+$ simultaneously.

On the other hand, taking the real part  of the form \eqref{quadr_form} one gets
  \begin{equation}\label{real_part_form}
{\mathfrak r}_{\varphi(z)} [u]  :=  \RE {F}_{\varphi(z)} [u] =
\int_{\R_+}|u'(x)|^2dx + \RE{\varphi(z)}\cdot|u(0)|^2, \quad  u\in
W^1_2(\mathbb R_+), \quad z \in \mathbb C_l.
     \end{equation}
If  $\varphi(\cdot)\in S_+$, then this  form is nonnegative for each $z\in \C_l,$ hence
$F_\varphi(\cdot)\in S_+(\cH)$. This example demonstrates Proposition
\ref{prop_form_stability}  applied to the real part of $F_\varphi(\cdot)$ in place  of
its  imaginary part: the form ${\mathfrak r}_{\varphi(z)}$ is closed for each $z\in
\C_l,$
$$
\cH_{+,r}(z) = \cH_{+,r}(z_0) = W^1_2(\mathbb R_+), \ z\in\mathbb
C_l, \quad \text{and}\quad \mathcal D_r[F] = W^1_2(\mathbb R_+).
$$
Moreover, the operator associated with  the form
\eqref{real_part_form} is given by
\[
F_{\varphi(z),R}(z) u = -\frac{d^2u}{dx^2},\quad  \dom
(F_{\varphi(z),R}(z)) = \{u\in W^2_2(\mathbb R_+) : u'(0) =
(\RE\varphi(z))u(0)\}, \  z \in \C_l.
\]
In accordance with  Definition \ref{def_imaginary_part} (in fact, with its real
counterpart) $F_{\varphi(z),R}(\cdot)$ is the real part of the function
$F_{\varphi(z)}(\cdot)\in S_+(\cH)$.

It is easily seen that $a\in\sigma_c(F_\varphi(z))$ for each $z\in \C\setminus \R_+$ and
$a\ge0.$ This fact correlates with Propositions \ref{prop_continuous_spec}
 for $z\in \C_+$ and  $z\in \C_l$, respectively.
       \end{example}
%
%
     \begin{example}
Let $\varphi(\cdot)$   and $\cH$ be as above.  Define  an operator-valued function
$G_\varphi(\cdot)$  by
\[
G_\varphi(z) f = -i\frac{d^2u}{dx^2},\qquad  \dom (G_\varphi(z)) =
\{u\in W^2_2(\mathbb R_+) :\ u'(0) = \varphi(z)u(0)\}, \quad z \in
\C_+.
\]
Clearly, $\rho(G_\varphi(z))\not = \emptyset$ for each  $z \in
\C_+.$ Furthermore,  the corresponding family of quadratic forms
have the description
    \begin{equation}\label{6.7}
{G}_{\varphi(z)} [u] =  i\int_{\R_+}|u'(x)|^2dx + \varphi(z)|u(0)|^2, \quad u\in \dom
{G}_{\varphi(z)} = \dom (G_\varphi(z)).
    \end{equation}
It follows that the form ${G}_{\varphi(z)} [\cdot]$  is dissipative
for each $z \in \C_+$, hence $G(\cdot)\in R(\cH)$  and $\mathcal
D(G) = W^2_{2,0}(\mathbb R_+)$ is dense in $\cH$. Taking  imaginary
part in \eqref{6.7} we get
$$
{\mathfrak t}_{\varphi(z)}[u] := \IM {G}_{\varphi(z)}[u] =
\int_{\R_+}|u'(x)|^2dx + \IM \varphi(z)|u(0)|^2, \quad \dom
{\mathfrak t}_{\varphi(z)} = \dom (G_\varphi(z)).
$$
This  form is   closable   and  its  closure  is  given by
$$
{\overline {\mathfrak t}}_{\varphi(z)}[u] := \IM {G}_{\varphi(z)}[u]
= \int_{\R_+}|u'(x)|^2dx + \IM \varphi(z)|u(0)|^2, \quad \dom
{\overline {\mathfrak t}}_{\varphi(z)} = W^1_2(\mathbb R_+),  \  z
\in \mathbb C_+.
$$
The latter is in accordance with Proposition  \ref{prop_form_stability}:
$$
\cH_+(z) = W^1_2(\mathbb R_+), \quad  z\in\mathbb C_+, \quad
\text{and}\quad  \mathcal D[G_{\varphi}] = W^1_{2}(\mathbb R_+).
$$
The operator associated with the form ${\overline {\mathfrak
t}}_{\varphi(z)}$ (the imaginary part of the operator
$G_{\varphi}(z)$ in the sense of Definition
\ref{def_imaginary_part}) is given by
\[
G_{\varphi(z),I} u = - \frac{d^2u}{dx^2},\quad  \dom
(G_{\varphi(z),I}) = \{u\in W^2_2(\mathbb R_+): u'(0) =  (\IM
\varphi(z)) u(0)\}, \  z \in \C_+.
\]
According to Proposition \ref{prop_Harmonic_cont_and_point_spec},
$a\in\sigma_c(G_{\varphi(z),I})$ for each $z\in \C_+$ and  $a\ge0.$
Moreover, by Proposition \ref{prop_continuous_spec},
$0\in\sigma_c(G_{\varphi(z)})$ for each $z\in \C_+.$
   \end{example}
%
%
     \begin{example}\label{ex_finite_int-val}
Let $\varphi(\cdot)$  be  a scalar $R$-function and $\cH =L^2(0, 1).$ Consider an
operator-valued function $G_\varphi(\cdot)$ given  by
\[
G_\varphi(z) u = -i\frac{d^2u}{dx^2},\quad  \dom (G_\varphi(z)) = \{u\in W^2_2(0,1) :\
u'(0) = \varphi(z)u(0),\  u(1)=0\}, \quad z \in \C_+.
\]
It is easily seen that $\rho(G_\varphi(z))\not = \emptyset$  for each  $z \in \C_+$ and
the operator $G_\varphi(z)$ has discrete spectrum. Moreover, the corresponding quadratic
form is
    \begin{equation}\label{6.7B}
{G}_{\varphi(z)} [u] =  \int_0^1 |u'(x)|^2\,dx + \varphi(z)|u(0)|^2, \quad u\in \dom
{G}_{\varphi(z)} = \dom (G_\varphi(z)).
    \end{equation}
Clearly, the form is dissipative, hence  $G(\cdot)\in R(\cH)$  and
$$
\mathcal D(G) := \cap_{z\in\mathbb C_+}\dom G(z) =  \{u\in
W^2_{2}(0,1): u(0)= u'(0)=u(1) =0\}
$$
is dense in $\cH$. Taking imaginary part in \eqref{6.7B}  one
obtains a nonnegative closable form ${\mathfrak
t}_{\varphi(z)}[\cdot]$ defined on $\dom (G_\varphi(z))$. Its
closure is given by
  $$
{\overline {\mathfrak t}}_{\varphi(z)}[u] := \IM {G}_{\varphi(z)}[u]
= \int_0^1 |u'(x)|^2dx + \IM \varphi(z)|u(0)|^2, \quad \dom
({\overline {\mathfrak t}}_{\varphi(z)}) = {\widetilde
W}^2_{2,0}(0,1), \  z \in \mathbb C_+,
  $$
where ${\widetilde W}^2_{2,0}(0,1) := \{u\in W^1_2(0,1) :\  u(1)=0\}$.

The latter is in accordance with Proposition  \ref{prop_form_stability}:
$$
\cH_+(z) = {\widetilde W}^2_{2,0}(0,1), \quad  z\in\mathbb C_+,
\quad \text{and}\quad \mathcal D[G_{\varphi}] = {\widetilde
W}^2_{2,0}(0,1).
$$

The operator associated with the form ${\overline {\mathfrak
t}}_{\varphi(z)}$ (the imaginary part of $G_{\varphi}(z)$) is given
by
\[
G_{\varphi,I}(z) u = - \frac{d^2u}{dx^2},\   \dom (G_{\varphi,I})(z) = \{u\in
W^2_2(\mathbb R_+): u'(0) -  (\IM \varphi(z)) u(0)= u(1)=0\}. 
\]
Since the spectrum of $G_{\varphi}(z)$ is discrete,  $\sigma_c(G_{\varphi}(z)) =
\sigma_c(G_{\varphi,I}(z))=\emptyset$ for each $z\in \C_+$. This fact is in accordance
with Propositions \ref{prop_continuous_spec} and
\ref{prop_Harmonic_cont_and_point_spec}.

Moreover, the estimate $s_j((G_{\varphi(z)})^{-1}) = O(j^{-2})$, $j\in \N$, holds for
each $z\in \C_+.$ This fact correlates with Remark  \ref{rem_general_ideal}.
     \end{example}


\begin{thebibliography}{DHMS1}



\bibitem{AmrPear04}
W.O. Amrein, D.B. Pearson, \textit{$M$-operator: a generalisation of
Weyl-Titchmarsh theory}, J. Comp. Appl. Math., 171 (2004), 1--26.


\bibitem{AI}
T.Ya.~Azizov and I.S.~Iokhvidov, \textit{Linear operators in spaces
with indefinite metric}, John Wiley and Sons, New York, 1989.

\bibitem{BeLa07}
J. Behrndt, M. Langer, \textit{Boundary value problems for elliptic
partial differential operators on bounded domains}, J. Funct. Anal.,
243 (2007), 536--565.

\bibitem{BDHdS2011}
J. Behrndt, V. Derkach, S. Hassi, and H.S.V. de Snoo, \textit{A
realization theorem for generalized Nevanlinna families}, Operators
and Matrices, 5, No. 4 (2011), 679--706.

\bibitem{BHdS2009}
J. Behrndt, S. Hassi, and H.S.V. de Snoo,
\textit{Boundary relations, unitary colligations, and functional models},
Complex Analysis and Operator Theory, 3 (2009), 57--98.

\bibitem{Ben}
C.~Benewitz, \textit{Symmetric relations on a Hilbert space}, Lect. Notes
Math., 280 (1972), 212--218.

\bibitem{Br}
M.S.~Brodskii, \textit{Triangular and Jordan representations of
linear operators}, Nauka, Moscow, 1968.

\bibitem{Co}
E.A.~Coddington, \textit{Extension theory of formally normal and symmetric
subspaces}, Mem. Amer. Math. Soc., 134 (1973), 1--80.

\bibitem{CDR}
B.~\v{C}urgus, A.~Dijksma, and T.T.~Read, \textit{The linearization of
boundary eigenvalue problems and reproducing kernel Hilbert
spaces}, Linear Algebra and Appl., 329 (2001), 97--136.

\bibitem{DHMS1}
V.A. Derkach, S. Hassi, M.M. Malamud, and H.S.V. de Snoo,
\textit{Generalized resolvents of symmetric operators and
admissibility}, Methods of Functional Analysis and Topology, 6, no.3
(2000), 24--55.

\bibitem{DHMS06}
V.A. Derkach, S. Hassi, M.M. Malamud, and H.S.V. de Snoo, \textit{Boundary
relations and Weyl families}, Trans. Amer. Math. Soc., 358 (2006),
5351--5400.

\bibitem{DHMS12}
V.A. Derkach, S. Hassi, M.M. Malamud, and H.S.V. de Snoo,
\textit{Boundary triplets and Weyl functions. Recent developments},
London Mathematical Society Lecture Notes, 404, (2012), 161--220.

\bibitem{DHM15}
V.A. Derkach, S. Hassi, M.M. Malamud, \textit{On a class of
generalized boundary triplets and form domain invariant Nevanlinna
functions}, (in preparation).

\bibitem{DM91}
V.A.~Derkach and M.M.~Malamud, \textit{Generalized resolvents and the
boundary value problems for hermitian operators with gaps}, J.
Funct. Anal., 95 (1991), 1--95.

\bibitem{DM95}
V.A.~Derkach and M.M.~Malamud, \textit{The extension theory of hermitian
operators and the moment problem}, J. Math. Sciences, 73 (1995),
141--242.
\bibitem {DM97}
        V.A.~Derkach and M.M.~Malamud,
        On some classes of Holomorphic Operator Functions
        with Nonnegative Imaginary Part, \textit{16th OT Conference
        Proceedings,  Operator theory, operator algebras and related topics}
        (Timisoara, 1996), 113-147, Theta Found., Bucharest, 1997.

\bibitem{Don74}
W.F.~Donoghue, \textit{Monotone matrix functions and analytic
continuation}, Springer-Verlag, Berlin-Heidelberg-New York, 1974.

\bibitem{FW}
P.A.~Fillmore and J.P.~Williams, \textit{On operator ranges}, Adv.
Math., 7 (1971), 254--281.


\bibitem{GesTs00}
F. Gesztesy, E. Tsekanovskii, \textit{On matrix-valued Herglotz
functions}, Math. Nachr. 218 (2000), 61–-138.

\bibitem{GK65}
 I.C. Gohberg, M.G. Krein, \textit{Introduction to the theory of linear non-selfadjoint operators in Hilbert space},  Nauka, Moscow 1965, 448 pp.

\bibitem{HdSSz2009}
S. Hassi, H.S.V. de Snoo, and F.H. Szafraniec, \textit{Componentwise
and Cartesian decompositions of linear relations},
Dissertationes Mathematicae 465, Polish Academy of Sciences,
Warszawa, 2009, 59 pp.


\bibitem{KacK}
I.S.~Kac and M.G.~Kre\u{\i}n, $R$-functions -- analytic functions
mapping the upper halfplane into itself, Supplement to the Russian
edition of F.V.~Atkinson, \textit{Discrete and continuous boundary
problems}, Mir, Moscow 1968 (Russian) (English translation: Amer.
Math. Soc. Transl. Ser. 2, 103 (1974), 1--18).

\bibitem{Kato}
T. Kato, \textit{Perturbation theory for linear operators}, Springer
Verlag, Berlin -- Heidelberg -- New York, 1966.
\bibitem{Kr46}
M.G.~Kre\u{\i}n, \textit{On the resolvents of an Hermitian operator
with the deficiency-index (m,m)}.  (Doklady) Acad. Sci. URSS (N.S.)
52, (1946), 651–654.

\bibitem{KosMMM}
A.S. Kostenko, M.M. Malamud, \textit{1--D Schr\"odinger operators
with local point interactions on a discrete set}, J. Differential
Equations 249 (2010), no. 2, 253--304.

\bibitem{KL71}
M.G.~Kre\u{\i}n and H.~Langer, \textit{On defect subspaces and generalized
resolvents of Hermitian operator in Pontryagin space}, Funkts.
Anal. i Prilozhen. 5, no.2 (1971), 59--71; ibid. 5 no. 3 (1971),
54--69 (Russian) (English translation: Funct. Anal. Appl., 5 (1971),
136--146; ibid. 5 (1971), 217--228).
\bibitem{KL73}
M.G.~Kre\u{\i}n and H.~Langer, \textit{\"Uber die $Q$-function eines
$\pi$-hermiteschen Operators in Raume $\Pi_{\kappa}$}, Acta. Sci.
Math. (Szeged), 34 (1973), 191--230. Siberian Math. J., 18 (1977),
728--746.

\bibitem{MalMal03}
M.M.~Malamud and S.M.~Malamud, \textit{Spectral theory of operator measures
in  Hilbert space}, St.~-Petersburg Math. Journal, 15, no. 3 (2003),
1-77.

\bibitem{MN2012}
M.M. Malamud and H. Neidhardt, \textit{Sturm-Liouville boundary
value problems with operator potentials and  unitary equivalence},
{J. Differential Equations},  \textbf{252} (2012), 5875--5922.


\bibitem{Ph0}
R.S.~Phillips, \textit{Dissipative operators and hyperbolic systems of
partial differential equations}, Trans. Amer. Math. Soc., 90
(1959), 192--254.


\bibitem{SNF}
B. Sz.-Nagy and C. Foias, \textit{Harmonic analysis of operators in
Hilbert space}, Budapest, 1967.

\bibitem{Post12}
O. Post, \textit{Boundary pairs associated with quadratic forms},
arXiv:1210.4707

\bibitem{Titch62}
 E.C. Titchmarsh, \textit{Eigenfunction expansions associated with second-order differential equations. Part I. Second Edition Clarendon Press}, Oxford 1962 vi+203 pp. 34.30
\end{thebibliography}
\end{document}